\documentclass[11pt]{article}

\usepackage{amsmath}
\usepackage{amssymb}
\newcommand{\R}{{\mathbb R}}
\newcommand{\Rp}{{\mathbb R^+}}

\newcommand{\sgn}{\mbox{\rm sgn}}

\newcommand{\N}{{\mathbb N}}
\newcommand{\LL}{{\mathbb L}}
\newcommand{\s}{{\bigcup_{\beta<1}{\cal S}^\beta}}
\newcommand{\h}{{\bigcup_{\beta<1}{\cal H}^\beta}}

\newcommand{\FF}{{\cal F}}
\newcommand{\HH}{{\cal H}}

\newcommand{\arrowp}{\mathop{\longrightarrow}_{P}}

\newcommand{\Rd}{{{{\mathbb R}^d}}}

\newtheorem{theorem}{\bf Theorem}[subsection]
\newtheorem{proposition}[theorem]{\bf Proposition}
\newtheorem{corollary}[theorem]{\bf Corollary}
\newtheorem{remark}[theorem]{\bf Remark}

\newenvironment{proof}{{\sc Proof.}}{\hfill $\Box$}

\newcommand{\nsubsection}{\setcounter{equation}{0}\subsection}

\begin{document}
\title{$\LL^p$ solutions of reflected BSDEs under
monotonicity condition}
\author{Andrzej Rozkosz and Leszek
S\l omi\'nski\footnote{Corresponding author.
Tel.: +48-566112954; fax: +48-566112987.}
\mbox{}\\[2mm]
{\small  Faculty of Mathematics and Computer Science,
Nicolaus Copernicus University},\\
{\small Chopina 12/18, 87-100 Toru\'n, Poland}}
\date{}
\maketitle
\begin{abstract}
We prove existence and  uniqueness of $\LL^p$ solutions,
$p\in[1,2]$, of reflected backward stochastic differential
equations with $p$-integrable data and generators satisfying the
monotonicity condition. We also show that the solution may be
approximated by the penalization method.  Our results are new even
in the classical case $p=2$.
\end{abstract}
{\em Keywords:} reflected backward stochastic differential
equations, monotonicity condition, $p$-integrable data.

\footnotetext{{\em Email addresses:} rozkosz@mat.uni.torun.pl (A.
Rozkosz), leszeks@mat.uni.torun.pl (L. S\l omi\'nski). }

\nsubsection{Introduction}

Nonlinear backward stochastic differential equations (BSDEs) were
considered for the first time by Pardoux and Peng \cite{pp}. In
the paper \cite{ekppq} by El Karoui et al. the so called reflected
BSDEs (RBSDEs) were introduced. By a solution of the RBSDE with
terminal value $\xi$, generator
$f:[0,T]\times\Omega\times\R\times\Rd\to\R$ and obstacle
$L=\{L_t,t\in[0,T]\}$ we understand a triple $(Y,Z,K)$ of
$(\FF_t)$ adapted processes such that
\begin{equation}
\label{eq1.1} \left\{
\begin{array}{l}
Y_t=\xi+\int^T_tf(s,Y_s,Z_s)\,ds -\int^T_tZ_s\,dW_s +K_T-K_t,\quad
t\in[0,T],\medskip\\
Y_t\ge L_t,\quad t\in[0,T],\medskip \\
K\mbox{ is nondecreasing, continuous, }K_0=0,\,\,
\int^T_0(Y_t-L_t)\,dK_t=0,
\end{array}
\right.
\end{equation}
where  $W$ is a standard $d$-dimensional Wiener process and
$(\FF_t)$ is the standard augmentation of the natural filtration
generated by $W$. It is assumed here that $\xi$ is $\FF_T$
measurable and $L$ is an $(\FF_t)$ progressively measurable
continuous process such that $L_T\le\xi$ a.s. Condition in
(\ref{eq1.1})${}_2$ says that the first component $Y$ of the
solution is forced to stay above $L$. The role of $K$ is to push
$Y$ upwards in order to keep it above $L$. We also require that
$K$ be minimal in the sense of (\ref{eq1.1})${}_3$, i.e. $K$
increases only when $Y=L$. Note that usual BSDEs may be considered
as special case of RBSDEs with $L\equiv-\infty$ (and $K\equiv0$).

In \cite{pp} it is proved  that if $\xi\in\LL^2,\quad
\int_0^T(f(s,0,0))^2\,ds\in\LL^1$ and $f$ is Lipschitz continuous
in both variables $y,z$ then there exists a unique solution
$(Y,Z)$ of BSDE with data $\xi,f$ such that $Y\in{\cal S}^2$,
$Z\in\HH^2$, i.e. $Y$ is continuous and adapted, $Z$ is
progressively measurable, and $Y^*_T\in\LL^2$, $
(\int^T_0|Z_t|^2\,dt)^{1/2}\in\LL^2$ (here and later on we use the
notation $X^*_t=\sup_{s\leq t}X_s$, $t\in[0,T]$). In \cite{ekppq}
existence and uniqueness of a solution $(Y,Z,K)$ of (\ref{eq1.1})
such that $Y,K\in{\cal S}^2$, $Z\in\HH^2$ is proved under the
additional assumption that $L^+=\max(L,0)\in{\cal S}^2$.

The assumptions on the data  in \cite{ekppq,pp} are sometimes too
strong for applications (see, e.g., \cite{de,ekpq} for
applications in economics and finance and \cite{bc,rs} for
applications to PDEs). Therefore many attempts have been made to
weaken the integrability conditions imposed in \cite{ekppq,pp} on
$\xi$ and $f$ or weaken the assumption that $f$ is Lipschitz
continuous. For instance, Briand and Carmona \cite{bc} and Pardoux
\cite{pa}  consider square-integrable solutions (i.e. $Y\in{\cal
S}^2$, $Z\in\HH^2)$ of BSDEs with generators which are Lipschitz
continuous with respect to $z$ while with respect to $y$ are
continuous and satisfy the monotonicity condition and the general
growth condition of the form $|f(t,y,z)|\le
|f(t,0,z)|+\varphi(|y|)$. In $\cite{bc}$ $\varphi$ is a polynom,
whereas in \cite{pa} an arbitrary positive continuous increasing
function. In El Karoui et. al. \cite{ekpq}  conditions ensuring
existence and uniqueness of $\LL^p$ solutions (i.e. $Y,Z\in{\cal
S}^p$, $Z\in\HH^p)$ for $p>1$ of BSDEs with Lipschitz continuous
generator with respect to both $y$ and $z$ are given. The
strongest results in this direction are given in Briand et al.
\cite{bdhps}, where $\LL^p$ solutions of BSDEs for $p\in[1,2]$ are
considered. It is proved there that in  case $p\in(1,2]$ if
\[
\xi\in\LL^p,\quad \int^T_0|f(s,0,0)|\,ds\in\LL^p,\quad
\forall_{r>0}\,\int_0^T\sup_{|y|\leq r}|f(s,y,0)-f(s,0,0)|\,ds
<+\infty
\]
and $f$ is Lipschitz continuous in $z$ and  continuous and
monotone in $y$  then there exists a unique $\LL^p$ solution.
Similar result is proved for $p=1$ in case $f$ does not depend on
$z$ and in the general case under some additional assumption
(assumption (H5) in Section \ref{sec5}). Finally, let us mention
that many papers are devoted  to BSDEs with quadratic growth
generators in $z$ (see, e.g., \cite{K} and the references given
there).

In Matoussi \cite{M} existence of square-integrable solutions of
RBSDEs with continuous generators satisfying the linear growth
condition is proved. Square-integrable solutions of RBSDEs under
monotonicity and the general  growth condition with respect to $y$
were considered by Lepeltier et al. in \cite{lmx}. In Hamad\`ene
and Popier \cite{hp} existence and uniqueness of $\LL^p$ solutions
of RBSDEs is proved in case $p\in(1,2)$ for $\xi\in\LL^p$,
$L^+\in{\cal S}^p$ and generators which are Lipschitz continuous
in $y$ and $z$ and satisfy the condition
$\int_0^T|f(s,0,0)|\,ds\in\LL^p$. Similar result for generators
satisfying the monotonicity condition and the linear growth
condition with respect to $y$ is proved in Aman \cite{A}. $\LL^1$
solutions of some generalized Markov type RBSDEs with random
terminal time are considered in \cite{rs}.

In the present paper we study $\LL^p$ solutions of RBSDEs of the
form (\ref{eq1.1}) for $p\in[1,2]$. Our main theorems on existence
and uniqueness of solutions may be summarized by saying that if
$\xi,f$ satisfy assumptions from \cite{bdhps} and the obstacle $L$
satisfies the assumptions
\[
L^{+,*}_T\in\LL^p,\quad\int_0^T|f(s,L^{+,*}_s,0)|ds\in\LL^p,
\]
then there exists a unique $\LL^p$ solution of (\ref{eq1.1}). It
is worth noting that as in \cite{bdhps} we do not assume that $f$
satisfies the general growth condition in $y$. Therefore our
results strengthen known results  proved in \cite{A,lmx} even in
the classical case $p=2$ (see Remark 4.4) and results proved in
\cite{A,hp} in case $p\in(1,2)$. We also show that the solution
$(Y,Z,K)$ to (\ref{eq1.1}) may be approximated by the penalization
method if $p\in(1,2]$ and if $p=1$ and $f$ is independent of $z$.
More precisely, if $p\in(1,2]$ then
\begin{equation}
\label{eq1.3} \|Y^n-Y\|_{{\cal S}^p}\to0,\quad \|Z^n-Z\|_{{\cal
H}^p}\to 0,\quad\|K^n-K\|_{{\cal S}^p}\to0,
\end{equation}
where $(Y^n,Z^n)$ is a solution of the BSDE
\begin{equation}
\label{eq1.2}
Y^n_t=\xi+\int_t^Tf(s,Y^n_s,Z^n_s)\,ds-\int_t^TZ^n_s\,dW_s
+K^n_T-K^n_t,\quad t\in[0,T]
\end{equation}
with
\[
K^n_t=n\int_0^t(Y^n_s-L_s)^-\,ds,\quad t\in[0,T].
\]
This  generalizes  and at the same time strengthens corresponding
result proved in \cite{lmx} in case $p=2$. In case $p=1$ we show
that (\ref{eq1.3}) holds in the spaces ${\cal S}^{\beta}$, ${\cal
H}^{\beta}$ with $\beta\in(0,1)$.

The paper is organized as follows. Section 2 contains basic
notation and  definitions. A useful a priori estimate for stopped
solutions of RBSDEs is also given. In Section 3 we prove main
estimates in  case $p\in(1,2]$. In Section 4 we apply the above
mentioned estimates  to prove convergence of penalization scheme
in case $p\in(1,2]$. Section 5 is devoted to the case where $p=1$.
For generator $f$ not depending on $z$ we give some a priori
estimates similar to those proved in case $p>1$ and we show
convergence of the penalization scheme. In the general case,
following \cite[Section 6]{bdhps} we prove existence and
uniqueness of solutions of (\ref{eq1.1}) under some additional
assumption on $f$. In this case the solution is a limit of
solutions of appropriately chosen RBSDEs  with generators not
depending on $z$.

\nsubsection{Notation and preliminary estimates}

Let $(\Omega,{\cal F},P)$ be a complete probability space.
$\LL^p$, $p>0$, is the space of random variables $X$ such that
$\|X\|_p=E(|X|^p)^{1\wedge1/p}<+\infty$. $X^*_t=\sup_{s\le
t}|X_s|$, $t\in[0,T]$. ${\cal S}^p$ denotes the set of adapted and
continuous processes $X$ such that $\|X\|_{{\cal
S}^p}=\|X^*_T\|_p<+\infty$. Let $W$ be a standard $d$-dimensional
Wiener process on $(\Omega, {\cal F}, P)$  and let $(\FF_t)$ be
the standard augmentation of the natural filtration generated by
$W$. ${\cal H}^p$ denotes the set of progressively measurable
$d$-dimensional processes $X$ such that $\|X\|_{{\cal H}^p}
=\|(\int_0^T|X_s|^2\,ds)^{1/2}\|_p<+\infty$. It is well known that
${\cal S}^p$  and ${\cal H}^p$ are Banach spaces for $p\geq1$. If
$p<1$ then $\LL^p$, ${\cal S}^p$ and ${\cal H}^p$ are complete
metric spaces with metrics defined by $\|\cdot\|_p$,
$\|\cdot\|_{{\cal S}^p}$ and $\|\cdot\|_{{\cal H}^p}$\,,
respectively.

We will assume that we are given an ${\cal F}_T$ measurable random
variable $\xi$, a generator
$f:[0,T]\times\Omega\times\R\times\R^d\to\R$ measurable with
respect to $Prog\otimes {\cal B}(\R)\otimes{\cal B}(\R^d)$, where
$Prog$ denotes the  $\sigma$-field of progressive subsets of
$[0,T]\times\Omega$ and a barrier $L$, which is an $({\cal F}_t)$
adapted  continuous process. We will always assume that $\xi\geq
L_T$. We will need the following assumptions on $f$.
\begin{enumerate}
\item[(H1)]There is $\lambda\ge0$ such that
$|f(t,y,z)-f(t,y,z')|\le\lambda|z-z'|$ for $t\in[0,T]$, $y\in\R$,
$z,z'\in\Rd$,
\item[(H2)]There is $\mu\in\R$ such that
$(y-y')(f(t,y,z)-f(t,y',z))\leq\mu(y-y')^2$ for $t\in[0,T]$,
$y,y'\in\R$, $z\in\Rd$.
\end{enumerate}
In (H1), (H2) and  in the sequel we understand that the
inequalities hold true  $P$-a.s..


\begin{proposition} \label{prop2.1}
Assume that $f$ satisfies (H1), (H2) and let $(Y,Z,K)$ be a
solution of (\ref{eq1.1}). Then for every $p>0$ there exists $C>0$
depending only on  $p$ and $\mu,\lambda, T$ such that for every
stopping time $\tau$ such that $\tau\le T$,
\begin{equation}
\label{eq2.01} E\big((\int_0^\tau|Z_s|^2ds)^{p/2}+K_\tau^p\big)
\le C E\Big((Y^*_\tau)^p+(L^{+,*}_\tau)^p
+(\int_0^\tau|f(s,L^{+,*}_s,0)|\,ds)^p\Big).
\end{equation}
\end{proposition}
\begin{proof}
Let $a\in\R$ and let $\tilde Y_t=e^{at}Y_t$, $\tilde
Z_t=e^{at}Z_t$, $\tilde K_t=\int_0^te^{as}dK_s$ and $\tilde
\xi=e^{aT}\xi$, $\tilde f(t,y,z)=e^{at}f(t,e^{-at}y,e^{-at}z)-ay$.
Observe that $(\tilde Y,\tilde Z, \tilde K)$ solves the RBSDE
\[
\tilde Y_t=\tilde \xi+\int_t^T\tilde f(s,\tilde Y_s,\tilde
Z_s)\,ds-\int_t^T\tilde Z_s\,dW_s +\tilde K_T-\tilde K_t,\qquad
t\in[0,T]
\]
with the reflecting barrier $\tilde L_t=e^{at}L_t$, and that there
exist constants $C_1,C_2>0$ depending only on $p,a,T$ such that
\begin{align*}
(\int_0^\tau|f(s,L^{+,*}_s,0)|ds)^p+(L^{+,*}_\tau)^p &\le
C_1\big((\int_0^\tau|\tilde f(s,\tilde L^{+,*}_s,0)|ds)^p
+(\tilde L^{+,*}_\tau)^p\big)\\
&\le
C_2\big((\int_0^\tau|f(s,L^{+,*}_s,0)|ds)^p+(L^{+,*}_\tau)^p\big).
\end{align*}
It follows that (\ref{eq2.01}) is satisfied if and only if it is
satisfied for the solution $(\tilde Y,\tilde Z,\tilde K)$ and the
data $\tilde\xi,\tilde f,\tilde L$ (with some constant $C$
depending also on $a$). Therefore choosing $a$ appropriately we
may assume that (H2) is satisfied with arbitrary but fixed
$\mu\in\R$. In the rest of the proof we will assume that $\mu=0$.
Moreover, without loss of generality we may and will assume that
$Y^*_\tau$, $L^{+,*}_\tau$, $\int_0^\tau|f(s,L^{+,*}_s,0)|ds$
$\in\LL^p$. Set $\tau_n=\inf\{t;\int_0^t|Z_s|^2ds\geq
n\}\wedge\tau$, $n\in\N$. Obviously $P(\tau_n=\tau)\nearrow1$. By
It\^o's formula applied to the continuous semimartingale
$Y-L^{+,*}$, for $n\in\N$ we have
\begin{align*}
&(Y_0-L^{+,*}_0)^2+\int_0^{\tau_n}|Z_s|^2\,ds
=(Y_{\tau_n}-L^{+,*}_{\tau_n})^2
+2\int_0^{\tau_n}(Y_s-L^{+,*}_s)f(s,Y_s,Z_s)\,ds\\
&\qquad-2\int_0^{\tau_n}(Y_s-L^{+,*}_s)Z_s\,dW_s
+2\int_0^{\tau_n}(Y_s-L^{+,*}_s)(dK_s+dL^{+,*}_s).
\end{align*}
Since $K$ is increasing only on the set $\{s:Y_s=L_s\}$,
\begin{align}
\label{eq2.1} \int_0^{\tau_n}(Y_s-L^{+,*}_s)\,dK_s
&\le \int_0^{\tau_n}(Y_s-L^{+,*}_s)
{\bf 1}_{\{Y_s>L^{+,*}_s\}} \,dK_s\nonumber\\
&=\int_0^{\tau_n}(Y_s-L^{+,*}_s){\bf 1}_{\{L_s>L^{+,*}_s\}}\,
dK_s=0.
\end{align}
By the above and (H1), (H2),
\begin{align*} \int_0^{\tau_n}|Z_s|^2\,ds
&\le(Y_{\tau_n}-L^{+,*}_{\tau_n})^2
+2\int_0^{\tau_n}|Y_s-L^{+,*}_s||f(s,L^{+,*}_s,0)|\,ds\\
&\quad+2\lambda\int_0^{\tau_n}|Y_s-L^{+,*}_s||Z_s|\,ds+
2|\int_0^{\tau_n}(Y_s-L^{+,*}_s)Z_s\,dW_s|\\
&\quad+2\sup_{t\leq \tau_n}|Y_t-L^{+,*}_t|L^{+,*}_{\tau_n}\\
&\le\sup_{t\leq \tau_n}(Y_t-L^{+,*}_t)^2+2\sup_{t\leq
\tau_n}|Y_t-L^{+,*}_t|\int_0^{\tau_n}|f(s,L^{+,*}_s,0)|\,ds\\
&\quad+\int_0^{\tau_n}\big(\frac{(2\lambda|Y_s-L^{+,*}_s|)^2}2
+\frac{|Z_s|^2}2\big)\,ds\\
&\quad+2|\int_0^{\tau_n}(Y_s-L^{+,*}_s)Z_s\,dW_s|
+\sup_{t\leq \tau_n}(Y_t-L^{+,*}_t)^2+(L^{+,*}_{\tau_n})^2\\
&\le(3+2\lambda^2T)\sup_{t\leq\tau_n}(Y_t-L^{+,*}_t)^2
+(\int_0^{\tau_n}|f(s,L^{+,*}_s,0)|\,ds)^2\\
&\quad+\frac12\int_0^{\tau_n}|Z_s|^2\,ds
+(L^{+,*}_{\tau_n})^2+2|\int_0^{\tau_n}(Y_s-L^{+,*}_s)Z_s\,dW_s|.
\end{align*}
Hence there is $C'>0$ such that
\[
\int_0^{\tau_n}|Z_s|^2\,ds\le
C'\big((Y^*_{\tau})^2+(L^{+,*}_{\tau})^2
+(\int_0^{\tau}|f(s,L^{+,*}_s,0)|ds)^2
+|\int_0^{\tau_n}(Y_s-L^{+,*}_s)Z_s\,dW_s|\big),
\]
which implies that for some $C'_p>0$,
\begin{align*}
(\int_0^{\tau_n}|Z_s|^2ds)^{p/2} &\leq
C'_p\Big((Y^*_{\tau})^p+(L^{+,*}_{\tau})^p\\
&\quad+(\int_0^{\tau}|f(s,L^{+,*}_s,0)|\,ds)^p
+|\int_0^{\tau_n}(Y_s-L^{+,*}_s)Z_s\,dW_s|^{p/2}\Big).
\end{align*}
By the Burkholder-Davis -Gundy inequality,
\begin{align*}
E|\int_0^{\tau_n}(Y_s-L^{+,*}_s)Z_s\,dW_s|^{p/2} &\le c_pE
(\int_0^{\tau_n}(Y_s-L^{+,*}_s)^2|Z_s|^2\,ds)^{p/4}\\
&\le c'_pE\big((Y^*_\tau)^p+(L^{+,*}_\tau)^p\big) +\frac12 E
(\int_0^{\tau_n}|Z_s|^2\,ds)^{p/2}.
\end{align*}
Putting together the last two estimates we see that there is $C>0$
such that
\[
E\big((\int_0^{\tau_n}|Z_s|^2\,ds)^{p/2}\big)\le C
E\Big((Y^*_\tau)^p+(L^{+,*}_\tau)^p
+(\int_0^\tau|f(s,L^{+,*}_s,0)|\,ds)^p\Big)
\]
for all $n\in\N$. Letting $n\to\infty$ and using Fatou's lemma we
conclude that
\begin{equation}
\label{eq2.2} E\big((\int_0^\tau|Z_s|^2\,ds)^{p/2}\big)\le C
E\Big((Y^*_\tau)^p+(L^{+,*}_\tau)^p
+(\int_0^\tau|f(s,L^{+,*}_s,0)|\,ds)^p\Big).
\end{equation}
In order to get estimates on $K$ we first observe that by
(\ref{eq1.1}),
\[
K_t= Y_0-Y_t-\int_0^tf(s,Y_s,Z_s)\,ds+\int_0^tZ_s\,dW_s,\quad
t\in[0,T].
\]
Hence $dK_s=-dY_s-f(s,Y_s,Z_s)\,ds +Z_s\,dW_s$. From this, (H1)
and the fact that  $K$  is increasing  only on the set
$\{s:L_s=Y_s\}$ it follows  that
\begin{align}
\label{eq2.3} K_\tau=\int_0^{\tau}{\bf 1}_{\{Y_s\leq
L^{+,*}_s\}}\,dK_s &= -\int_0^{\tau}{\bf 1}_{\{Y_s\leq
L^{+,*}_s\}}\,dY_s
-\int_{0}^{\tau}f(s,Y_s,Z_s){\bf 1}_{\{Y_s\le L^{+,*}_s\}}\,ds\nonumber\\
&\nonumber\quad +\int_{0}^{\tau}Z_s{\bf 1}_{\{Y_s\leq L^{+,*}_s\}}\,dW_s\\
 &\le-\int_0^{\tau}{\bf 1}_{\{Y_s\leq
L^{+,*}_s\}}\,dY_s -\int_{0}^{\tau}f(s,Y_s,0){\bf 1}_{\{Y_s\le
L^{+,*}_s\}}\,ds\nonumber\\
&\quad+\lambda T^{1/2}(\int_{0}^{\tau}|Z_s|^2\,ds)^{1/2}
+\int_{0}^{\tau}Z_s{\bf 1}_{\{Y_s\leq L^{+,*}_s\}}\,dW_s.
\end{align}
By the classical It\^o-Tanaka formula applied to the function
$g(x)=(x)^-=\max(-x,0)$ and the continuous semimartingale
$Y-L^{+,*}$,
\begin{align*}
-\int_0^{\tau}{\bf 1}_{\{Y_s\leq
L^{+,*}_s\}}\, dY_s&=-(Y_0-L^{+,*}_0)^-+(Y_{\tau}-L^{+,*}_{\tau})^-\\
&\quad-\int_0^{\tau}{\bf 1}_{\{Y_s\leq L^{+,*}_s\}}\,dL^{+,*}_s
-\frac12 L^0_{\tau}(Y-L^{+,*})\\
&\le Y^*_{\tau}+L^{+,*}_{\tau},
\end{align*}
where $L^0(Y-L^{+,*})$ denotes the usual local time of $Y-L^{+,*}$
at $0$. On the other hand, by (H2),
\[
-f(s,Y_s,0){\bf 1}_{\{Y_s\leq
L^{+,*}_s\}}\leq-f(s,L^{+,*}_s,0){\bf 1}_{\{Y_s\leq L^{+,*}_s\}}
\leq|f(s,L^{+,*}_s,0)|.
\]
From the above we deduce that there is $C_p>0$ such that
\begin{align*}
E(K_\tau)^p&\le  C_p E\Big((Y^*_{\tau})^p+(L^{+,*}_{\tau})^p
+(\int_0^{\tau}|f(s,L^{+,*}_s,0)|\,ds)^p\\
&\quad+(\int_{0}^{\tau}|Z_s|^2\,ds)^{p/2}+|\int_{0}^{\tau}Z_s{\bf
1}_{\{Y_s\leq L^{+,*}_s\}}\,dW_s|^p\Big).
\end{align*}
Combining this with (\ref{eq2.2}) and using the
Burkholder-Davis-Gundy we get (\ref{eq2.01}).
\end{proof}

\nsubsection{Main estimates in the case $p>1$}
\label{sec3}

Let $g:\R\to\R$ be a difference of two convex functions and let
$X$ be a continuous semimartingale. We will use the following form
of the It\^o-Tanaka formula
\begin{equation}
\label{eq3.1}
g(X_t)=g(X_0)+\int_0^t\frac12(g'_-+g'_+)(X_s)dX_s+\frac12\int_\R
\tilde L^a_t(X)g''(da)
\end{equation}
(see \cite[Exercise VI.1.25]{RY}). Here  $\tilde L^a(X)$ denotes
the symmetric local time of $X$ at $a\in\R$ and $g''(da)$ is a
measure determined by the second derivative of $g$ in the sense of
distributions. Note that $\tilde L^a(X)$ is a unique increasing
process such that
\begin{equation}
\label{eq3.2}
|X_t-a|=|X_0-a|+\int_0^t\sgn(X_s-a)dX_s+\tilde L^a_t(X),
\end{equation}
where $\sgn(x)$ is equal $1$ if $x>0$, $-1$ if $x<0$  and $0$ if
$x=0$ (see \cite[Exercise VI.1.25]{RY}). One can observe that
$\tilde L^a(X)=(L^a(X)+L^{a-}(X))/2$, where $L^a(X)$ denotes the
usual local time of  $X$ at $a$. If $g''$ is absolutely
continuous, i.e. if $g''(da)=g''(a)da$, then by the occupation
times formula, $\int_\R \tilde
L^a_t(X)g''(da)=\int_0^tg''(X_s)d[X]_s$. In this section we will
apply (\ref{eq3.1}) to functions of the form $g(x)=|x|^p$ or
$g(x)=((x)^+)^p$. If  $p>1$ then in both cases the second
derivative of $g$ is absolutely continuous. Therefore if $p>1$
then the backward It\^o-Tanaka formula has the form
\begin{equation}
\label{eq3.3} g(X_t)+\frac12\int_t^Tg''(X_s)\,d[X]_s
=g(X_T)-\int_t^T\frac12(g'_-+g'_+)(X_s)\,dX_s.
\end{equation}

We can now prove basic a priori estimate and comparison result for
$\LL^p$ solutions of (\ref{eq1.1}).

\begin{proposition}\label{prop2}
Assume that $f$ satisfies (H1), (H2) and let $(Y,Z,K)$ be a
solution of (\ref{eq1.1}) such that $Y\in{\cal S}^p$ for some
$p>1$. There exists $C>0$ depending only on $p$ and $\mu,\lambda,
T$ such that
\[
E\big(Y^*_{T}\big)^p\leq C
E\Big(|\xi|^p+(L^{+,*}_T)^p
+(\int_0^T|f(s,L^{+,*}_s,0)|\,ds)^p\Big).
\]
\end{proposition}
\begin{proof}
We follow the proof of \cite[Proposition 3.2]{bdhps}. The
reasoning used at the beginning of the proof of Proposition
\ref{prop2.1} shows that we may assume that $\mu=-\lambda^2/(p-1)$
and $\xi$, $L^{+,*}_T,\int_0^T|f(s,L^{+,*}_s,0)|\,ds\in\LL^p$. By
(\ref{eq3.3}),
\begin{align*}
&|Y_t-L^{+,*}_t|^p +\frac{p(p-1)}2\int_t^T|Y_s-L^{+,*}_s|^{p-2}
{\bf 1}_{\{Y_s\ne L^{+,*}_s\}}|Z_s|^2\,ds\\
&\qquad=|\xi-L^{+,*}_T|^p
+p\int_t^T|Y_s-L^{+,*}_s|^{p-1}\,\sgn(Y_s-L^{+,*}_s)f(s,Y_s,Z_s)\,ds\\
&\qquad\quad+p\int_t^T|Y_s-L^{+,*}_s|^{p-1}\,\sgn(Y_s-L^{+,*}_s)(dK_s+dL^{+,*}_s)\\
&\qquad\quad-p\int_t^T|Y_s-L^{+,*}_s|^{p-1}\,\sgn(Y_s-L^{+,*}_s)
Z_s\,dW_s.
\end{align*}
By (H2) and the fact that $K$ is increasing only on the set
$\{s:Y_s=L_s\}$,
\[
\sgn(Y_s-L^{+,*}_s)f(s,Y_s,Z_s)\le
|f(s,L^{+,*}_s,0)|+\mu|Y_s-L^{+,*}_s|+\lambda |Z_s|
\]
and
\[
\sgn(Y_s-L^{+,*}_s)\,dK_s\le 0.
\]
Hence
\begin{align*}
&|Y_t-L^{+,*}_t|^p +\frac{p(p-1)}2\int_t^T|Y_s-L^{+,*}_s|^{p-2}
{\bf 1}_{\{Y_s\ne L^{+,*}_s\}}|Z_s|^2\,ds\\
&\qquad\le|\xi-L^{+,*}_T|^p
+p\int_t^T|Y_s-L^{+,*}_s|^{p-1}(|f(s,L^{+,*}_s,0)|\,ds+dL^{+,*}_s)\\
&\qquad\quad+p\mu\int_t^T|Y_s-L^{+,*}_s|^p\,ds
+p\lambda\int_t^T|Y_s-L^{+,*}_s|^{p-1}|Z_s|\,ds\\
&\qquad\quad-p\int_t^T|Y_s-L^{+,*}_s|^{p-1}\,\sgn(Y_s-L^{+,*}_s)
Z_s\,dW_s.
\end{align*}
Since
\begin{align*}
p\lambda|Y_s-L^{+,*}_s|^{p-1}|Z_s|
\le\frac{p\lambda^2}{p-1}|Y_s-L^{+,*}_s|^{p}
+\frac{p(p-1)}4|Y_s-L^{+,*}_s|^{p-2} {\bf 1}_{\{Y_s\ne
L^{+,*}_s\}}|Z_s|^2
\end{align*}
for $s\in[0,T]$, we have
\begin{equation}
\label{eq3.4}
|Y_t-L^{+,*}_t|^p +\frac{p(p-1)}2\int_t^T|Y_s-L^{+,*}_s|^{p-2}
{\bf 1}_{\{Y_s\ne L^{+,*}_s\}}|Z_s|^2\,ds\le X-M_t,
\end{equation}
where
\[
X=|\xi-L^{+,*}_T|^p
+p\int_0^T|Y_s-L^{+,*}_s|^{p-1}(|f(s,L^{+,*}_s,0)|\,ds+dL^{+,*}_s)\]
and
\[
M_t= \int_0^t|Y_s-L^{+,*}_s|^{p-1}\sgn(Y_s-L^{+,*}_s)
Z_s\,dW_s,\quad t\in[0,T].
\]
Since $Y\in{\cal S}^p$ and, by Proposition \ref{eq2.1}, $Z\in{\cal
H}^p$, applying Young's inequality we obtain
\begin{align*}
EX&\le E|\xi-L^{+,*}_T|^p+E(\sup_{t\le T}
|Y_t-L^{+,*}_t|^{p-1}(\int_0^T|f(s,L^{+,*}_s,0)|\,ds+L^{+,*}_T)\\
&\le E|\xi-L^{+,*}_T|^p+\frac{p-1}{p}E\sup_{t\le T}|Y_t-L^{+,*}_t|^{p}\\
&\quad+\frac{2^{p-1}}{p}
E(\int_0^T|f(s,L^{+,*}_s,0)|\,ds)^p+(L^{+,*}_T)^p)<+\infty
\end{align*}
and
\begin{align*}
E([M]_T^{1/2})&\le E(\sup_{t\leq T}
|Y_t-L^{+,*}_t|^{p-1}(\int_0^T|Z_s|^2\,ds)^{1/2})\\
&\le \frac{(p-1)2^{p-1}}{p}E((Y^*_T)^p
+(L^{+,*}_T)^p)+\frac1{p}E(\int_0^T |Z_s|^2\,ds)^{p/2}<+\infty.
\end{align*}
In particular, $M$ is a uniformly integrable martingale and hence,
by (\ref{eq3.4}),
\begin{equation}
\label{eq3.05} \frac{p(p-1)}{2}E\int_0^T|Y_s-L^{+,*}_s|^{p-2} {\bf
1}_{\{Y_s\ne L^{+,*}_s\}}|Z_s|^2\,ds\leq EX.
\end{equation}
From (\ref{eq3.4}), (\ref{eq3.05}), the Burholder-Davis-Gundy
inequality and the definition of $M$ it follows that there is
$c_p$ such that
\begin{align*}
&E\sup_{t\leq T}|Y_t-L^{+,*}_t|^p\le EX+c_pE[M]^{1/2}_T\\
&\qquad\le EX+c_pE(\sup_{t\leq T}
|Y_t-L^{+,*}_t|^{p}\int_0^T|Y_s-L^{+,*}_s|^{p-2}({\bf
1}_{\{Y_s\ne L^{+,*}_s\}}|Z_s|^2\,ds)^{1/2}\\
&\qquad\leq EX+\frac12E\sup_{t\le T}
|Y_t-L^{+,*}_t|^{p}+\frac{c^2_p}2E\int_0^T|Y_s-L^{+,*}_s|^{p-2}
{\bf 1}_{\{Y_s\ne L^{+,*}_s\}}|Z_s|^2\,ds\\
&\qquad\le(1+\frac{c^2_p}{p(p-1)})EX+ \frac12E\sup_{t\leq T}
|Y_t-L^{+,*}_t|^{p}.
\end{align*}
By the above, the definition of $X$ and Young's inequality,
\begin{align*}
&E\sup_{t\le T}|Y_t-L^{+,*}_t|^p\le c_p'EX\\
&\qquad \le c_p'(E|\xi-L^{+,*}_T|^p
+pE\int_0^T|Y_s-L^{+,*}_s|^{p-1}(|f(s,L^{+,*}_s,0)|\,ds+dL^{+,*}_s)\\
&\qquad \le c_p'E|\xi-L^{+,*}_T|^p+ pc_p'E\sup_{t\le T}
|Y_t-L^{+,*}_t|^{p-1}(\int_0^T|f(s,L^{+,*}_s,0)|\,ds+L^{+,*}_T)\\
&\qquad\le \frac12E\sup_{t\leq T}|Y_t-L^{+,*}_t|^{p}+c_p''(
E|\xi-L^{+,*}_T|^p+E(\int_0^T|f(s,L^{+,*}_s,0)|\,ds+L^{+,*}_T)^p).
\end{align*}
Hence
\[
E\sup_{t\leq T}|Y_t-L^{+,*}_t|^p\le 2c_p''(E|\xi-L^{+,*}_T|^p
+(\int_0^T|f(s,L^{+,*}_s,0)|\,ds+L^{+,*}_T)^p),
\]
from which the required estimate for $Y^*_T$ follows.
\end{proof}

\begin{proposition}
\label{prop3.2} Let $(Y,Z,K)$ be a solution of (\ref{eq1.1}) with
$f$ satisfying (H1), (H2) and let $(Y',Z',K')$ be a solution of
(\ref{eq1.1}) with data $\xi'$, $f'$, $L'$ such that
$\xi\leq\xi'$, $f(t,Y'_t,Z'_t)\leq f'(t,Y'_t,Z'_t)$ and $L_t\leq
L'_t$, $t\in[0,T]$. If $Y,Y'\in{\cal S}^p$ for some $p>1$ then
$Y_t\leq Y'_t$, $t\in[0,T]$.
\end{proposition}
\begin{proof} Assume that $\mu=-\lambda^2/(p-1)$. Then by (H1),
(H2),
\begin{align*}
&((Y_s-Y'_s)^+)^{p-1}(f(s,Y_s,Z_s)-f(s,Y'_s,Z'_s))\\
&\qquad\le
-\frac{\lambda^2}{p-1}((Y_s-Y'_s)^+)^p+\lambda((Y_s-Y'_s)^+)^{p-1}|Z_s-Z'_s|
\end{align*}
for $s\in[0,T]$. Hence, by (\ref{eq3.3}),
\begin{align*}
&((Y_t-Y'_t)^+)^p+\frac{p(p-1)}2\int_t^T((Y_s-Y'_s)^+)^{p-2}{\bf
1}_{\{Y_s> Y'_s\}}|Z_s-Z'_s|^2\,ds\\
&\quad=((\xi-\xi')^+)^p+\int_t^T((Y_s-Y'_s)^+)^{p-1}\sgn(Y_s-Y'_s)
(f(s,Y_s,Z_s)-f(s,Y'_s,Z'_s))\,ds\\
&\qquad+p\int_t^T((Y_s-Y'_s)^+)^{p-1}\sgn(Y_s-Y'_s)(dK_s-dK'_s)\\
&\qquad-p\int_t^T((Y_s-Y'_s)^+)^{p-1}\sgn(Y_s-Y'_s)(Z_s-Z'_s)\,dW_s\\
&\quad\le((\xi-\xi')^+)^p
-\frac{p\lambda^2}{p-1}\int_t^T((Y_s-Y'_s)^+)^{p}\,ds\\
&\qquad+p\lambda\int_t^T((Y_s-Y'_s)^+)^{p-1}|Z_s-Z'_s|\,ds
-p\int_t^T((Y_s-Y'_s)^+)^{p-1}(Z_s-Z'_s)\,dW_s.
\end{align*}
Since
\begin{align*}
&p\lambda((Y_s-Y'_s)^+)^{p-1}|Z_s-Z'_s|\\
&\quad\le\frac{p\lambda^2}{p-1}((Y_s-Y'_s)^+)^{p}
+\frac{p(p-1)}4((Y_s-Y'_s)^+)^{p-2}{\bf 1}_{\{Y_s>Y'_s\}}
|Z_s-Z'_s|^2,
\end{align*}
it follows that
\begin{align}
\label{eq3.06}
&((Y_t-Y'_t)^+)^p+\frac{p(p-1)}4\int_t^T((Y_s-Y'_s)^+)^{p-2} {\bf
1}_{\{Y_s>Y'_s\}}|Z_s-Z'_s|^2\,ds\nonumber\\
&\qquad\le-p\int_t^T((Y_s-Y'_s)^+)^{p-1}(Z_s-Z'_s)\,dW_s.
\end{align}
Finally, as in the proof of Proposition \ref{prop2} one can check
that $M$ defined by
\[
M_t=\int_0^t ((Y_s-Y'_s)^+)^{p-1}(Z_s-Z'_s)\,dW_s,\quad
t\in[0,T]
\]
is a uniformly integrable martingale. Therefore from
(\ref{eq3.06}) it follows that $E((Y_t-Y'_t)^+)^p=0$, $t\in[0,T]$.
\end{proof}
\medskip

By repeating arguments from the proof of  Proposition
\ref{prop3.2} one can obtain the following version  of the
comparison theorem for  nonreflected BSDEs.
\begin{corollary}
\label{cor3.3} Let $(Y,Z)$ be a solution of nonreflected
(\ref{eq1.1}) (i.e., where $L=-\infty$  and $K=0$) with $f$
satisfying (H1), (H2) and let $(Y',Z')$ be a solution of
nonreflected (\ref{eq1.1}) with data $\xi'$, $f'$,  such that
$\xi\leq\xi'$ and $f(t,Y'_t,Z'_t)\leq f'(t,Y'_t,Z'_t)$,
$t\in[0,T]$. If $Y,Y'\in{\cal S}^p$ for some $p>0$ then $Y_t\le
Y'_t$, $t\in[0,T]$.
\end{corollary}

Note that Corollary \ref{cor3.3} generalizes the comparison result proved in
\cite{pa}  for square-integrable solutions of nonreflected BSDEs.

\nsubsection{Existence and uniqueness of solutions in the case
$p>1$} \label{sec4}

We begin with a general uniqueness result.

\begin{proposition}
\label{corr} If $f$ satisfies (H1), (H2) then there is at most one
solution $(Y,Z,K)$ of (\ref{eq1.1}) such that $Y\in{\cal S}^p$ for
some $p>1$.
\end{proposition}
\begin{proof}
Follows from Proposition \ref{prop3.2}.
\end{proof}
\medskip

The problem of existence of solutions is more delicate. In the
present section we will assume additionally that
\begin{enumerate}
\item[(H3)] (a) $E|\xi|^p<+\infty$,

(b) For every $t\in[0,T]$ and $z\in\Rd$, $y\mapsto f(t,y,z)$ is
continuous,

(c) $E(\int_0^T|f(s,0,0)|\,ds)^p<+\infty$,

(d) for every $r>0$, $\int_0^T\sup_{|y|\le r}
|f(s,y,0)-f(s,0,0)|\,ds <+\infty$,
\item[(H4)] (a) $E(L^{+,*}_T)^p<+\infty$,

(b) $E(\int_0^T|f(s,L^{+,*}_s,0)|ds)^p<+\infty$.
\end{enumerate}

From Theorem 4.2 and  Remark 4.3 in \cite{bdhps} one can deduce
that under (H1)--(H3), (H4a) for every $n\in\N$ there exists a
unique solution $Y^n\in{\cal S}^p$, $Z^n\in{\cal H}^p$ of the BSDE
(\ref{eq1.2}).

\begin{proposition}
\label{prop5} Let $f$ satisfy (H1), (H2) and let $(Y^n,Z^n,K^n)$,
$n\in\N$, be a solution of (\ref{eq1.2}). Then for every $p>0$
there exists $C>0$ depending only on $p$ and $\mu,\lambda, T$ such
that for every stopping time $\tau\le T$ and $n\in\N$,
\[
E\big((\int_0^\tau|Z^n_s|^2\,ds)^{p/2}+(K^n_\tau)^p\big)\le C
E\Big((Y^{n,*}_\tau)^p+(L^{+,*}_\tau)^p
+(\int_0^\tau|f(s,L^{+,*}_s,0)|\,ds)^p\Big).
\]
\end{proposition}
\begin{proof} The proof is similar to that of Proposition \ref{prop2.1}.
Set $\tilde Y^n_t=e^{at}Y^n_t$, $\tilde Z^n_t=e^{at}Z^n_t$ and
$\tilde L^n_t=e^{at}L^n_t$, $\tilde \xi=e^{aT}\xi$, $\tilde
f(t,y,z)=e^{at}f(t,e^{-at}y,e^{-at}z)-ay$. Then $(\tilde
Y^n,\tilde Z^n)$ solves the BSDE
\[
\tilde Y^n_t=\tilde \xi+\int_t^T\tilde f(s,\tilde Y^n_s,\tilde
Z^n_s)\,ds-\int_t^T\tilde Z^n_s\,dW_s +\tilde K^n_T-\tilde
K^n_t,\qquad t\in[0,T]
\]
with the penalization term
\[
\tilde K^n_t=\int_0^te^{as}dK^n_s=n\int_0^t(\tilde Y^n_s-\tilde
L_s)^-\,ds\quad t\in[0,T].
\]
Therefore without loss of generality we may assume that $\mu=0$.
Since $K^n$ is increasing only on the set $\{s:Y^n_s<L_s\}$,
\[
\int_0^{t}(Y^n_s-L^{+,*}_s)\,dK^n_s\leq\int_0^{t}(Y_s-L^{+,*}_s){\bf
1}_{\{Y^n_s>L^{+,*}_s\}}\,dK^n_s=0
\]
and
\begin{align*}
K^n_\tau=\int_0^{\tau}{\bf 1}_{\{Y^n_s\leq L^{+,*}_s\}}\,dK^n_s
&\le Y^{n,*}_{\tau}+L^{+,*}_{\tau}+\int_0^{\tau}|f(s,L^{+,*}_s,0)|\,ds\\
&\quad+\lambda T^{1/2}(\int_{0}^{\tau}|Z^n_s|^2\,ds)^{1/2}
+\int_{0}^{\tau}Z^n_s{\bf 1}_{\{Y^n_s\leq L^{+,*}_s\}}\,dW_s.
\end{align*}
To get the desired estimate it suffices now to repeat step by step
arguments from the proof of Proposition \ref{prop2.1}, the only
difference being in using the above estimates involving $K^n$
instead of (\ref{eq2.1}), (\ref{eq2.3}).
\end{proof}

\begin{proposition}
\label{prop6} Let assumptions (H1)--(H4) hold and let
$(Y^n,Z^n,K^n)$ be a solution of (\ref{eq1.2}). Then for every
$p>1$ there exists $C>0$ depending only on $p$ and $\mu,\lambda,
T$ such that for every $n\in\N$,
\begin{align*}
&E\big((Y^{n,*}_T)^p+(\int_0^T|Z^n_s|^2\,ds)^{p/2}+(K^n_T)^p\big)\\
&\qquad\le
CE\Big(|\xi|^p+(L^{+,*}_T)^p+(\int_0^T|f(s,L^{+,*}_s,0)|\,ds)^p\Big).
\end{align*}
\end{proposition}
\begin{proof}
Since
\begin{align*}
&p\int_t^T|Y^n_s-L^{+,*}_s|^{p-1}\sgn(Y^n_s-L^{+,*}_s)\,dK^n_s\\
&\qquad\leq p\int_t^T|Y^n_s-L^{+,*}_s|^{p-1}{\bf
1}_{\{Y^n_s>L^{+,*}_s\}}{\bf 1}_{\{Y^n_s<L_s\}}\,dK^n_s =0,
\end{align*}
applying the It\^o-Tanaka formula to the function $g(x)=|x|^p$ and
the semimartingale $Y^n-L^{+,*}$ we can estimate $E(Y^{n,*}_T)^p$
in much the same way as in Proposition \ref{prop2} (by the results
from \cite{bdhps} we know that $Y^n\in{\cal S}^p$, $n\in\N$).
Therefore the desired result follows from Proposition \ref{prop5}
with $\tau=T$.
\end{proof}

\begin{theorem}
\label{tw1}
Assume that (H1)--(H4)  are satisfied. If $(Y^n,Z^n,K^n)$,
$n\in\N$, is a solution of BSDE (\ref{eq1.2}), then
\[
\|Y^n-Y\|_{{\cal S}^p}\to0,\quad \|Z^n-Z\|_{{\cal H}^p}\to
0,\quad\|K^n-K\|_{{\cal S}^p}\to0,
\]
where $(Y,Z,K)$ is a unique solution of the reflected BSDE
(\ref{eq1.1}) such that $Y\in{\cal S}^p$, $Z\in{\cal H}^p$ and
$K\in {\cal S}^p$.
\end{theorem}
\begin{proof}
Without loss of generality we may assume that $\mu=0$. Let
$(Y^n,Z^n,K^n)$ be a solution of  (\ref{eq1.2}). By Corollary
\ref{cor3.3}, $Y^n_t\leq Y^{n+1}_t$, $n\in\N$. Therefore for every
$t\in[0,T]$ there exists $Y_t$ such that $Y^n_t\nearrow Y_t$. The
rest of the proof is divided into 3 steps.
\medskip\\
{\em Step 1.} We show that $Y$  is a c\`adl\`ag process. To see
this let us first note that for every $t\in[0,T]$ there exists
$V_t$ such that
\[
0\leq V^n_t=\sup_{s\leq t}(Y^n_s-Y^1_s)\nearrow V_t.
\]
By Fatou's lemma, $Y_t,V_t$ are finite. Indeed, $|Y_t|\leq
V_t+Y^{1,*}_t$, $t\in[0,T]$, and by Proposition \ref{prop6},
\[
E(V_T)\leq \liminf_{n\to\infty}E(V^n_T)\leq2\sup_nE(Y^{n,*}_T)\leq
2\sup_n\|Y^{n,*}_T\|_p<\infty.
\]
Therefore $V$ is a progressively measurable nondecreasing process.
Since the filtration $(\FF_t)_{t\geq0}$ is right-continuous,
setting
\[
V'_t=\inf_{t'>t} V_{t'},\quad t\in[0,T)\quad\mbox{\rm and}\quad
V'_T=V_T
\]
we get a progressively measurable c\`adl\`ag process $V'$.
Obviously  $V_t\leq V'_t$, so $Y^{n,*}_t\leq V'_t+Y^{1,*}_t$, $
t\in[0,T]$, $n\in\N$ . For $k\in\N$ set now
\begin{equation}
\label{eq4.5} \tau_k=\inf\{t;\min(V'_t+Y^{1,*}_t,L^{+,*}_t,
\int_0^t|f(s,L^{+,*}_s,0)|\,ds)>k\}\wedge T.
\end{equation}
Clearly $\tau_k\leq\tau_{k+1}$, $k\in\N$, and $P(\tau_k=
T)\nearrow1$. Since $Y^n$ is a continuous process,
\[
Y^{n,*}_{\tau_k}=Y^{n,*}_{\tau_k-}\le
V'_{\tau_k-}+Y^{1,*}_{\tau_k-}\le \max(k,c),\quad k\in\N,
\]
where $c=\sup_n(Y^n_0)^+$ with the convention that
$Y^{n,*}_{0-}=Y^{n,*}_0$, $V'_{0-}=V'_{0}$ ($c$ is a nonnegative
constant because $Y^n_0$, $n\in\N$, are deterministic and by
Proposition \ref{prop6}, $|Y^n_0|\le CE(|\xi|^p+(L^{+,*}_T)^p
+(\int_0^T|f(s,L^{+,*}_s,0)|\,ds)^p)$ for $n\in\N$). Moreover,
$L^{+,*}_{\tau_k}\leq \max(k,c)$ and
$\int_0^{\tau_k}|f(s,L^{+,*}_s,0)|\,ds\leq k$. Putting $p=p'>2$
and $\tau=\tau_k$ in Proposition \ref{prop5} we get
\begin{equation}
\label{eq4.6} \sup_nE\big((\int_0^{\tau_k}|Z^n_s|^2\,ds)^{p'/2}
+(K^n_{\tau_k})^{p'}\big)\leq 3C\,\max(k,c)^{p'}<+\infty,
\end{equation}
and consequently,
\begin{equation}
\label{eq4.7}
\sup_nE|\int_0^{\tau_k}f(s,Y^n_s,Z^n_s)\,ds|^{p'}<+\infty.
\end{equation}
Since $f$ is  Lipschitz continuous with respect to  $z$,
\[
\int_0^{t}f(s,Y^n_s,Z^n_s)\,ds
=\int_0^{t}h^n_s\,ds+\int_0^{t}f(s,Y^n_s,0)\,ds,\quad t\in[0,T],
\]
where $h^n_s=(f(s,Y^n_s,Z^n_s)-f(s,Y^n_s,0)){\bf
1}_{\{|Z^n_s|>0\}}=C^n_s|Z^n_s|$ and $C^n$ is a one-dimensional
progressively measurable process bounded by $\lambda$. By
(\ref{eq4.6}),
$\sup_nE(\int_0^{\tau_k}(h^n_s)^2\,ds)^{p'/2}\leq+\infty$. Since
the sequences $\{{\bf 1}_{\{\cdot\leq\tau_k\}}h^n\}_{n\in\N}$\,,
$\{{\bf 1}_{\{\cdot\leq\tau_k\}}Z^n\}_{n\in\N}$ are bounded in
${\cal H}^2$, there exist a subsequence $(n')\subset(n)$, a
one-dimensional progressively measurable process $h$ and a
$d$-dimensional progressively measurable process $Z$ such that
${\bf 1}_{\{\cdot\leq\tau_k\}}h^{n'}\rightarrow h$ and ${\bf
1}_{\{\cdot\le\tau_k\}}Z^{n'}\to Z$  weakly in ${\cal H}^2$, i.e.
for any one-dimensional $h'\in{\cal H}^2$,
\begin{equation}
\label{eq4.9} E\int_0^T{\bf 1}_{\{s\leq\tau_k\}}h^{n'}_sh'_s\,ds
\to E\int_0^Th_sh'_s\,ds.
\end{equation}
and for any $d$-dimensional process $Z'\in{\cal H}^2$,
\begin{equation}
\label{eq4.8}
E\int_0^T{\bf 1}_{\{s\leq\tau_k\}}Z^{n'}_sZ'_s\,ds\to
E\int_0^TZ_sZ'_s\,ds,
\end{equation}
From (\ref{eq4.8}) and (\ref{eq4.9}) it follows that $h,Z$ are
equal to $0$ on the set $\{s>\tau_k\}$. Moreover, for every
stopping time $\sigma\leq\tau_k$,
\begin{equation}\label{eq4.10}
\int_0^\sigma h^{n'}_s\,ds\to\int_0^\sigma h_s\,ds, \quad
\int_0^\sigma Z^{n'}_s\,dW_s\to\int_0^\sigma Z_s\,dW_s
\end{equation}
weakly in $\LL^2$. Indeed, in order to prove the first convergence
in (\ref{eq4.10}) let us first observe that replacing $h'$ by
${\bf 1}_{\{s\leq\sigma\}}h'$ in (\ref{eq4.9}) shows that for
every $h'\in{\cal H}^2$,
\[
E\int_0^T{\bf 1}_{\{s\leq\sigma\}}h^{n'}_sh'_s\,ds\rightarrow
E\int_0^T {\bf 1}_{\{s\leq\sigma\}}h_sh'_s\,ds.
\]
Let $Y\in\LL^2$. Then $h'=E(Y|\FF_{\cdot})\in {\cal H}^2$ since
$E\int_0^T|E(Y|\FF_s)|^2ds\le TEY^2<+\infty$. Hence, by
(\ref{eq4.9}) and Fubini's theorem,
\begin{align*}
E(Y\int_0^T{\bf 1}_{\{s\leq\sigma\}}h^{n'}_s\,ds)
&=E(\int_0^TY{\bf 1}_{\{s\leq\sigma\}}h^{n'}_s\,ds
=\int_0^TE(Y{\bf 1}_{\{s\leq\sigma\}}h^{n'}_s)\,ds\\
&=\int_0^TE(E(Y|\FF_s){\bf 1}_{\{s\leq\sigma\}}h^{n'}_s)\,ds\\
&\rightarrow\int_0^TE(E(Y|\FF_s){\bf 1}_{\{s\le\sigma\}}
h_s)\,ds=E(Y\int_0^T{\bf 1}_{\{s\leq\sigma\}} h_s\,ds),
\end{align*}
which means that $\int_0^\sigma h^{n'}_s\,ds
\rightarrow\int_0^\sigma h_s\,ds $ weakly in $\LL^2$. By the
representation theorem, $Y=\int^T_0Z'_s\,dW_s$ for some
$d$-dimensional process $Z'\in{\cal H}^2$. Hence, by
(\ref{eq4.8}),
\[
E(Y\int_0^\sigma Z^{n'}_s\,dW_s)=\int_0^T{\bf
1}_{\{s\leq\sigma\}}Z^{n'}_sZ'_s\,ds\rightarrow \int_0^T{\bf
1}_{\{s\leq\sigma\}}Z_sZ'_s\,ds=E(Y\int_0^\sigma Z_s\,dW_s),
\]
which proves the second convergence in (\ref{eq4.10}). By
(H3b)--(H3d) and the Lebesgue dominated convergence theorem, for
every stopping time $\sigma\leq\tau_k$,
$\int_0^{\sigma}f(s,Y^n_s,0)\,ds\to\int_0^{\sigma}f(s,Y_s,0)\,ds$
$P$-a.s., and by (\ref{eq4.6}) and (\ref{eq4.7}), the last
convergence holds in $\LL^2$, too. Since  $Y^n_\sigma\nearrow
Y_\sigma$ in $\LL^2$ as well, for every stopping time $\sigma\leq
\tau_k$,
\begin{equation}\label{eq4.11}
Y_\sigma=Y_0-\int_0^\sigma f(s,Y_s,0)ds-\int_0^\sigma
h_s\,ds+\int_0^\sigma Z_s\,dW_s-K_\sigma,
\end{equation} where
$K_\sigma $ is a weak limit in $\LL^2$ of $\{K^n_{\sigma}\}$. From
the proof of the  monotone limit theorem for BSDE (see Peng
\cite[Lemma 2.2]{peng}) it follows that $Y$ is c\`adl\`ag and $K$
is nondecreasing c\`adl\`ag on the stochastic interval
$[0,\tau_k]$. Since $P(\tau_k= T)\nearrow1$, it follows that
$P$-almost all trajectories of $Y$ are c\`adl\`ag  on the whole
interval $[0,T]$.
\medskip\\
{\em Step 2.} We  show that $Y_t\geq L_t$, $t\in[0,T]$  and
$(Y^n-L)^{-,*}_T\to0$ $P$-a.s. By (H3a), (H4) and Proposition
\ref{prop6} there is $C>0$ such that
$E(\int_0^T(Y^n_s-L_s)^-ds)^p\leq C/{n^p}$. Hence, by Fatou's
lemma,
\[
E\int_0^T(Y_s-L_s)^-\,ds\le
\liminf_{n\to\infty}E\int_0^T(Y^n_s-L_s)^-\,ds=0,
\]
which implies that $\int_0^T(Y_s-L_s)^-ds=0$. Since $Y-L$ is a
c\`adl\`ag process, $(Y_t-L_t)^-=0$ for $t\in[0,T)$ and hence
$Y_t\geq L_t$ for $t\in[0,T)$. Moreover, $Y_T=Y^n_T=\xi\geq L_T$.
Hence $(Y^n_t-L_t)^-\searrow 0$ for $t\in[0,T]$ and by Dini's
theorem, $(Y^n-L)^{-,*}_T\to0$ $P$-a.s.
\medskip\\
{\em Step 3.} We show that $\{(Y^n,Z^n,K^n)\}_{n\in\N}$ converges
in ${\cal S}^p\times{\cal H}^p\times{\cal S}^p$ to $(Y,Z,K)$,
where  $(Y,Z,K)$ is a  unique solution of (\ref{eq1.1}). Let
$\{\tau_k\}$ be a sequence of stopping times defined in {\em Step
1}. By It\^o's formula, (H1) and (H2) with $\mu=0$,
\begin{align*}
&(Y^n_{t\wedge \tau_k}-Y^m_{t\wedge\tau_k})^2
+\int_{t\wedge\tau_k}^{\tau_k}|Z^n_s-Z^m_s|^2\,ds\\
&\quad=(Y^n_{\tau_k}-Y^m_{\tau_k})^2+
2\int_{t\wedge\tau_k}^{\tau_k}(Y^n_s-Y^m_s)
(f(s,Y^n_s,Z^n_s)-f(s,Y^m_s,Z^m_s))\,ds\\
&\qquad+2\int_{t\wedge\tau_k}^{\tau_k}(Y^n_s-Y^m_s)(dK^n_s-dK^m_s)
-2\int_{t\wedge\tau_k}^{\tau_k}(Y^n_s-Y^m_s)(Z^n_s-Z^m_s)\,dW_s\\
&\quad\le(Y^n_{\tau_k}-Y^m_{\tau_k})^2+
2\lambda\int_{t\wedge\tau_k}^{\tau_k}|Y^n_s-Y^m_s||Z^n_s-Z^m_s|ds\\
&\qquad+2(Y^n-L)^{-,*}_{\tau_k}K^m_{\tau_k}
+2(Y^m-L)^{-,*}_{\tau_k}K^n_{\tau_k}
-2\int_{t\wedge\tau_k}^{\tau_k}(Y^n_s-Y^m_s)(Z^n_s-Z^m_s)\,dW_s\\
&\quad\le(Y^n_{\tau_k}-Y^m_{\tau_k})^2+
2\lambda^2\int_{t\wedge\tau_k}^{\tau_k}(Y^n_s-Y^m_s)^2ds
+2(Y^n-L)^{-,*}_{\tau_k}K^m_{\tau_k}+2(Y^m-L)^{-,*}_{\tau_k}K^n_{\tau_k}\\
&\qquad+\frac12\int_{t\wedge\tau_k}^{\tau_k}|Z^n_s-Z^m_s|^2\,ds
-2\int_{t\wedge\tau_k}^{\tau_k}(Y^n_s-Y^m_s)(Z^n_s-Z^m_s)\,dW_s.
\end{align*}
Hence
\begin{align*}
E\int_{0}^{\tau_k}|Z^n_s-Z^m_s|^2\,ds &\le
2E(Y^n_{\tau_k}-Y^m_{\tau_k})^2
+4\lambda^2E\int_{0}^{\tau_k}(Y^n_s-Y^m_s)^2\,ds\\
&\quad+4E(Y^n-L)^{-,*}_{\tau_k}K^m_{\tau_k}
+4E(Y^m-L)^{-,*}_{\tau_k}K^n_{\tau_k}.
\end{align*}
By Fubini's theorem,
\[
E\int_{0}^{\tau_k}(Y^n_s-Y^m_s)^2ds=\int_0^T E(Y^n_s{\bf
1}_{\{s\le\tau_k\}}-Y^m_s{\bf 1}_{\{s\le\tau_k\}})^2\,ds,
\]
which converges to $0$ as $m,n\to \infty$. By {\em Step 2} and
(\ref{eq4.6}),
\[
E(Y^n-L)^{-,*}_{\tau_k}K^m_{\tau_k}
\le\|(Y^n-L)^{-,*}_{\tau_k}\|_{2}\|K^m_{\tau_k}\|_2
\leq\|(Y^n-L)^{-,*}_{\tau_k}\|_{2}\sup_m\|K^m_{\tau_k}\|_2\,,
\]
which converges to 0 as $n\to\infty$. Similarly,
$E(Y^m-L)^{-,*}_{\tau_k}K^n_{\tau_k}\to0$ as $m\to\infty$. Since
$E(Y^n_{\tau_k}-Y^m_{\tau_k})^2\to0$  as $m,n\to\infty$, it is
clear that $\{{\bf 1}_{\{s\leq\tau_k\}}Z^n_s\}_{n\in\N}$ is a
Cauchy sequence in ${\cal H}^2$. Let $Z^{(k)}$  denote its limit.
By using standard arguments based on the Burkholder-Davis-Gundy
inequality one can show that in fact $E\sup_{t\leq \tau_k}
|Y^n_t-Y^m_t|^2\to0$ as $n,m\to\infty$, which implies that
$\sup_{t\leq \tau_k}|Y^n_t-Y_t|\displaystyle\arrowp0$ (here
$\displaystyle\arrowp$ stands for the convergence in probability
$P$). Since $P(\tau_k=T) \nearrow1$,
\begin{equation}
\label{eq4.12} \sup_{t\leq T}|Y^n_t-Y_t|\arrowp0,
\end{equation}
and consequently, $Y$ has continuous trajectories. Similarly, if
we set $\tau_0=0$, $Z_t=Z^{(k)}_t$, $t\in(\tau_{k-1},\tau_k]$,
$k\in\N$ and $Z_0=0$, then
\begin{equation}
\label{eq4.13}\int_{0}^{T}|Z^n_s-Z_s|^2\,ds\arrowp0.
\end{equation}
To see this let us fix $\varepsilon>0$. By Chebyshev's inequality,
for each $k\in\N$,
\begin{align*}
P(\int_{0}^{T}|Z^n_s-Z_s|^2\,ds>\varepsilon)&\le
P(\int_{0}^{\tau_k}|Z^n_s-Z_s|^2\,ds>\varepsilon,\,T
=\tau_k)+P(T>\tau_k)\\
&\le\varepsilon^{-2}E\int_0^{\tau_k}|Z^n_s-Z^{(k)}_s|^2\,ds
+P(T>\tau_k).
\end{align*}
Since we know that ${\bf 1}_{\{s\le\tau_k\}}(Z^n_s-Z^{(k)}_s)
\rightarrow0$ in ${\cal H}^2$, it follows that
\[
\limsup_{n\to\infty}P(\int_{0}^{T}|Z^n_s-Z_s|^2\,ds>\varepsilon)
\le P(T>\tau_k),
\]
which proves (\ref{eq4.13}) since $P(T>\tau_k)\searrow0$.  By
(H3c) and (H3d), $|f(s,Y^n_s,Z^n_s)|\leq g_k(s)+\lambda|Z^n_s|$,
$s\leq\tau_k$, where  $g_k$ is an integrable function. Hence, by
(H3b) and (\ref{eq4.13}),
\[
\int_0^{t\wedge\tau_k}f(s,Y^n_s,Z^n_s)\,ds
\arrowp\int_0^{t\wedge\tau_k}f(s,Y_s,Z_s)\,ds
\]
for every $k\in\N$. Letting $k\to\infty$ shows that we can omit
$\tau_k$ in the upper limit of integration. From the above  we
deduce that
\[
K^n_t\arrowp
K_t=Y_0-Y_t-\int_0^tf(s,Y_s,Z_s)\,ds+\int_0^tZ_s\,dW_s,\quad
t\in[0,T],
\]
where $K$ is a continuous nondecreasing process such that $K_0=0$.
It is clear that in fact,
\begin{equation}
\label{eq4.14}\sup_{t\leq T}|K^n_t-K_t|\arrowp0.
\end{equation}
By the above and (\ref{eq4.12}),
$0\geq\int_0^T(Y^n_t-L_t)\,dK^n_t\arrowp\int_0^T(Y_t-L_t)\,dK_t$,
which when combined with {\em Step 2} implies  that
$\int_0^T(Y_t-L_t)\,dK_t=0$. Putting together the facts mentioned
above we deduce that $(Y,Z,K)$ is a solution of the reflected BSDE
(\ref{eq1.1}).

In order to complete the proof we have to  show that
$(Y^n,Z^n,K^n)$ converges to $(Y,Z,K)$ in ${\cal S}^p\times{\cal
H}^p\times{\cal S}^p$. To see this let us first observe that by
(\ref{eq4.12})--(\ref{eq4.14}), Proposition \ref{prop6} and
Fatou's lemma, $(Y,Z,K)\in{\cal S}^p\times{\cal H}^p\times{\cal
S}^p$, which together with Proposition \ref{corr} implies that
$(Y,Z,K)$ is a unique solution of (\ref{eq1.1}) in ${\cal
S}^p\times{\cal H}^p\times{\cal S}^p$. Since
\[
\sup_n\sup_{t\leq T}|Y^n_s|\leq\sup_{t\leq T}|Y_s| +\sup_{t\leq T}
|Y^1_s|,
\]
applying the Lebesgue dominated convergence theorem shows that
$\|Y^n-Y\|_{{\cal S}^p}\to0$. Now,  we are going to estimate
 $Z^n-Z$ in the norm of ${\cal H}^p$. By
It\^o's formula,
\begin{align*}
&\int_{0}^{T}|Z^n_s-Z_s|^2\,ds=(Y^n_{0}-Y_{0})^2+
2\int_{0}^{T}(Y^n_s-Y_s)(f(s,Y^n_s,Z^n_s)-f(s,Y_s,Z_s))\,ds\\
&\qquad+2\int_{0}^{T}(Y^n_s-Y_s)(dK^n_s-dK_s)
-2\int_{0}^{T}(Y^n_s-Y_s)(Z^n_s-Z_s)\,dW_s.
\end{align*}
Hence, by (H1) and (H2) with $\mu=0$,
\begin{align*}
\int_{0}^{T}|Z^n_s-Z_s|^2\,ds &\le
2(Y^n_{0}-Y_{0})^2+4\lambda^2E\int_{0}^{T}(Y^n_s-Y_s)^2\,ds\\
&\quad+4(Y^n-L)^{-,*}_{T}K_{T}
+4\int_{0}^{T}(Y^n_s-Y_s)(Z^n_s-Z_s)\,dW_s.
\end{align*}
By {\em Step 2}, $(Y^n-L)^{-,*}_{T}\le
(Y^n-Y)^{*}_{T}+(Y-L)^{-,*}_{T}=(Y^n-Y)^{*}_{T}$. Hence
\begin{align*}
(\int_{0}^{T}|Z^n_s-Z_s|^2\,ds)^{p/2}&\le C((Y^n-Y)^*_T)^p
+2((Y^n-Y)^{*}_{T}K_{T})^{p/2}\\
&\quad+2|\int_{0}^{T}(Y^n_s-Y_s)(Z^n_s-Z_s)\,dW_s|^{p/2}.
\end{align*}
Using the Burkholder-Davis-Gundy inequality we deduce from the
above that
\[
E(\int_{0}^{T}|Z^n_s-Z_s|^2ds)^{p/2}\leq
C'(E(Y^n-Y)^*_T)^p+\|(Y^n-Y)^{*}_{T}\|_{p}\|K_{T}\|_p)\to0.
\]
On the other hand, there is  $C>0$ depending only  on $\lambda$
and $T$ such that
\begin{align*}
&\|\int_0^\cdot (f(s,Y^n,Z^n_s)-f(s,Y_s,Z_s)\,ds\|_{{\cal S}^p} \\
&\qquad\le\|\int_0^\cdot(f(s,Y^n,Z_s)-f(s,Y_s,Z_s)\,ds\|_{{\cal
S}^p} +C\|Z^n-Z\|_{{\cal H}^p}.
\end{align*}
By monotonicity of the mapping $y\mapsto f(s,y,z)$,
\[
\sup_n\sup_{t\leq T}|\int_0^tf(s,Y^n,Z_s)-f(s,Y_s,Z_s)\,ds|\le U,
\]
where $U=\sup_{t\leq T}|\int_0^t (f(s,Y_s,Z_s)\,ds|+\sup_{t\le T}
|\int_0^t(f(s,Y^1_s,Z_s)\,ds|\in\LL^p$. Hence, by the Lebesgue
dominated convergence theorem,
\[
\|\int_0^\cdot(f(s,Y^n,Z_s)-f(s,Y_s,Z_s)\,ds\|_{{\cal S}^p}\to0.
\]
Finally, putting together all the above convergences it is clear
that  $\|K^n-K\|_{{\cal S}^p}\to0$ and the proof of Theorem
\ref{tw1} is complete.
\end{proof}

\begin{remark}
{\rm Let us remark that if $f$ satisfies (H3c) and the general
increasing growth condition considered in \cite{lmx,pa}, i.e.
\begin{equation} \label{eq4.15} |f(t,y,0)|\leq
|f(t,0,0)|+\varphi(|y|),\quad t\in[0,T],\,y\in\R,
\end{equation}
where $\varphi:\Rp\to\Rp$ is a deterministic continuous increasing
function, and if $\varphi(L^{+,*}_T)$ $\in\LL^p$ then condition
(H4b) is satisfied. Moreover, if we assume  (H3c) and that
(\ref{eq4.15}) holds true for some measurable $\varphi:\Rp\to\Rp$
such that $\int_0^T \varphi(L^{+,*}_s)\,ds\in\LL^p$,  then (H4b)
is satisfied and the conclusion of Theorem \ref{tw1} is still in
force. Therefore Theorem \ref{tw1} generalizes and strengthens the
corresponding results of \cite{lmx} proved under condition
(\ref{eq4.15}) in case $p=2$ only. }
\end{remark}

\begin{corollary}\label{cor3}
Under the assumptions of Proposition \ref{prop3.2}, if moreover
$L=L'$ and  $\xi,f,L$ and $\xi'$, $f'$, $L$ satisfy (H3) and (H4),
$dK_s\geq dK'_s$.
\end{corollary}
\begin{proof}
By Proposition \ref{prop3.2},
\[
K^n_{t_2}-K^n_{t_1}=n\int_{t_1}^{t_2}(Y^n_s-L_s)^-\,ds \geq n
\int_{t_1}^{t_2}(Y'^n_s-L_s)^-\,ds =K'^n_{t_2}-K'^n_{t_1}
\]
for every $n\in\N$ and $0\le t_1\le t_2\le T$. Since
$\|K^n-K\|_{{\cal S}^p}\rightarrow0$ and $\|K'^n-K'\|_{{\cal
S}^p}\rightarrow0$ by Theorem \ref{tw1}, it follows that
$K_{t_2}-K_{t_1}\ge K'_{t_2}-K'_{t_1}$ for every $0\le t_1\le
t_2\le T$, which proves the desired result.
\end{proof}

\begin{remark}
\label{rem4.6} {\rm Of course (H2) is satisfied if $f$ is
Lipschitz continuous with respect to $y$, i.e. when
\begin{equation}\label{eq4.16}
|f(t,y,z)-f(t,y',z)|\leq C|y-y'|,\quad
t\in[0,T],\,y,y'\in\R,\,z\in\Rd
\end{equation}
for some $C\ge0$. Moreover, (\ref{eq4.16}) together with (H3c),
(H4a) imply (H4b).
Therefore conclusions of Theorem \ref{tw1} and Corollary
\ref{cor3} hold true if (H1), (H3a), (H3c), (H4a) and
(\ref{eq4.16}) are satisfied. Thus, Theorem \ref{tw1} and
Corollary \ref{cor3} strengthen the corresponding stability and
comparison results for (\ref{eq1.1}) proved in \cite{hp}, where
(\ref{eq4.16}) is assumed. }
\end{remark}

\nsubsection{$\LL^1$ solutions of reflected BSDEs}
\label{sec5}

Throughout this section we will assume that $p=1$ in  conditions
(H3), (H4).

Let us recall that a process $X$ belongs to the class ${\cal D}$
if the family of random variables $\{X_\sigma; \sigma\,\mbox{\rm
stopping time},\, \sigma\leq T\}$ is uniformly integrable. In
\cite[page 90]{dm} it is observed that the space of continuous
(c\`adl\`ag), adapted processes from ${\cal D}$ is complete under
the norm $\|X\|_{\cal D}=\sup\{E|X_\sigma|; \sigma\,\mbox{\rm
stopping time, }\sigma\leq T\}$.

First we consider the case where $f$ does not depend on  $z$.

\begin{proposition}
\label{prop7} Assume that $f$ satisfies (H2) and does not depend
on $z$,  and let $(Y,Z,K)$ be a solution of (\ref{eq1.1}) such
that $Y\in{\cal D}$.
\begin{enumerate}
\item[\rm(i)]There exists $C>0$ depending only on
$\mu,T$ such that
\[
\|Y\|_{\cal D} \le C E\big(|\xi|+L^{+,*}_T
+\int_0^T|f(s,L^{+,*}_s)|\,ds\big).
\]
\item[\rm(ii)]For every $\beta\in(0,1)$ there exists $C>0$ depending
only on $\beta,\mu,T$ such that
\[
E\big((Y^*_{T})^\beta +(\int_0^T|Z_s|^2\,ds)^{\beta/2}
+K^\beta_T\big)\leq C \big(E\big(|\xi|+L^{+,*}_T
+\int_0^T|f(s,L^{+,*}_s)|\,ds\big)\big)^\beta.
\]
\end{enumerate}
\end{proposition}
\begin{proof} We may and will assume that $\mu=0$. Let
$\tau_n=\inf\{t;\int_0^t|Z_s|^2ds\geq n\}\wedge T$, $n\in\N$. By
(\ref{eq3.2}),
\begin{align}
\label{eq5.1}
&|Y_{\sigma\wedge\tau_n}-L^{+,*}_{\sigma\wedge\tau_n}|+\tilde
L^0_{\sigma\wedge\tau_n}(Y-L^{+,*})=|Y_{\tau_n}-L^{+,*}_{\tau_n}|
+\int_{\sigma\wedge\tau_n}^{\tau_n}\sgn(Y_s-L^{+,*}_s)f(s,Y_s)\,ds
\nonumber\\
&\qquad+\int_{\sigma\wedge\tau_n}^{\tau_n}
\sgn(Y_s-L^{+,*}_s)(dK_s+dL^{+,*}_s)
-\int_{\sigma\wedge\tau_n}^{\tau_n}\sgn(Y_s-L^{+,*}_s)Z_s\,dW_s.
\end{align}
Since
\[
\int_{\sigma\wedge\tau_n}^{\tau_n}\sgn(Y_s-L^{+,*}_s)(dK_s+dL^{+,*}_s)
\le L^{+,*}_{\tau_n}- L^{+,*}_{\sigma\wedge\tau_n}
\]
and, by (H2),
\[
\int_{\sigma\wedge\tau_n}^{\tau_n}\sgn(Y_s-L^{+,*}_s)f(s,Y_s)\,ds
\le\int_{\sigma\wedge\tau_n}^{\tau_n}|f(s,L^{+,*}_s)|\,ds,
\]
it follows from (\ref{eq5.1}) that
\[
|Y_{\sigma\wedge\tau_n}|\leq |Y_{\tau_n}| +2L^{+,*}_{\tau_n}
+\int_{\sigma\wedge\tau_n}^{\tau_n}|f(s,L^{+,*}_s)|\,ds
-\int_{\sigma\wedge\tau_n}^{\tau_n}\sgn(Y_s-L^{+,*}_s) Z_s\,dW_s.
\]
Conditioning with respect to ${\cal F}_{\sigma\wedge\tau}$ and
then letting $n\to\infty$ we deduce from the above that
\begin{equation}
\label{eq5.2} |Y_\sigma|\leq E\big(|\xi|+2L^{+,*}_{T}
+\int_{0}^{T}|f(s,L^{+,*}_s)|\,ds|{\cal F}_\sigma\big),
\end{equation}
which implies (i).

By (\ref{eq5.2}) and \cite[Lemma 6.1]{bdhps},
\[
E(Y^*_{T})^\beta\leq (1-\beta)^{-1} \big(E\big(|\xi|+2L^{+,*}_T
+\int_0^T|f(s,L^{+,*}_s)|\,ds\big)\big)^\beta
\]
for every $\beta\in(0,1)$. Therefore (ii) follows from
(\ref{eq2.01}) with $\tau=T$.
\end{proof}
\medskip

let us note that by Proposition \ref{prop7}, if $(X,Z,K)$
satisfies (\ref{eq1.1}) and $Y\in{\cal D}$ then $Z\in\h$ and
$K\in\s$.

\begin{proposition}
\label{prop8} Under the assumptions of Proposition \ref{prop7}
there exists at most one solution  $(Y,Z,K)$ of (\ref{eq1.1}) such
that $Y\in{\cal D}$.
\end{proposition}
\begin{proof} Without loss of generality we may assume that $\mu=0$.
Let $(Y,Z,K)$, $(Y',Z',K')$ be two solutions of (\ref{eq1.1}).
Then from the  It\^o-Tanaka formula, (H2) and the inequality
\begin{equation}
\label{eq4d} \sgn(Y_s-Y'_s)(dK_s-dK'_s)\leq0
\end{equation}
it follows that for every $t\in[0,T]$,
\begin{align*}
&|Y_t-Y'_t|+\tilde L^0_t(Y-Y')
=\int_t^T\sgn(Y_s-Y'_s)(f(s,Y_s)-f(s,Y'_s))\,ds\\
&\qquad+\int_t^T\sgn(Y_s-Y'_s)(dK_s-dK'_s)-\int_t^T\sgn(Y_s-Y'_s)
(Z_s-Z'_s)\,dW_s\\
&\qquad\quad\le -\int_t^T\sgn(Y_s-Y'_s)(Z_s-Z'_s)\,dW_s.
\end{align*}
By using the fact that $Y,Y'\in{\cal D}$, stopping at
$\tau_n=\inf\{t;\int_0^t|Z_s-Z'_s|^2\,ds\geq n\}\wedge T$ and then
letting $n\rightarrow\infty$ we deduce from the above that
$E|Y_t-Y'_t|=0$, $t\in[0,T]$, i.e. $Y=Y'$. Consequently,
$\int_0^t(Z_s-Z'_s)\,dW_s=(K_t-K'_t)$, $t\in[0,T]$, which implies
that $Z=Z'$ and $K=K'$.
\end{proof}

\begin{proposition}
\label{prop9} Let $(Y,Z,K)$ be a solution of (\ref{eq1.1}) with
$f$ not depending on $z$ and satisfying (H2), and let $(Y',Z',K')$
be a solution of (\ref{eq1.1}) with data $\xi'$, $f'$, $L'$ such
that $\xi\leq\xi'$, $f'$ does not depend on $z$, $f(t,Y'_t)\leq
f'(t,Y'_t)$ and $L_t\leq L'_t$, $t\in[0,T]$. If $Y,Y'\in{\cal D}$
then $Y_t\leq Y'_t$, $t\in[0,T]$.
\end{proposition}
\begin{proof}
Assume that $\mu=0$ and observe that by (\ref{eq3.1}), (H2) and
(\ref{eq4d}),
\begin{align*}
&(Y_t-Y'_t)^+
+\frac{1}{2}\tilde L^0_t(Y-Y')
=\int_t^T{\bf 1}_{\{Y_s>Y'_s\}}(f(s,Y_s)-f(s,Y'_s))\,ds\\
&\qquad+\int_t^T{\bf 1}_{\{Y_s>Y'_s\}}(dK_s-dK'_s) -\int_t^T{\bf
1}_{\{Y_s>Y'_s\}} (Z_s-Z'_s)\,dW_s\\
&\qquad\quad\leq -\int_t^T{\bf 1}_{\{Y_s>Y'_s\}} (Z_s-Z'_s)\,dW_s.
\end{align*}
From this as in the proof of Proposition \ref{prop8} we deduce
that $E(Y_t-Y'_t)^+=0$, $t\in[0,T]$, i.e. $Y_t\leq Y'_t$,
$t\in[0,T]$.
\end{proof}

\begin{theorem}\label{tw2}
Assume that $f$ does not depend on $z$ and (H2)--(H4) are
satisfied. If $(Y^n,Z^n,K^n)$, $n\in\N$, is a solution of BSDEs
(\ref{eq1.2}) then for every $\beta\in(0,1)$,
\[
\|Y^n-Y\|_{{\cal S}^\beta}\to0,\quad \|Z^n-Z\|_{{\cal H}^\beta}\to
0,\quad \|K^n-K\|_{{\cal S}^\beta}\to0,
\]
where $(Y,Z,K)$ is a unique solution of the reflected BSDEs
(\ref{eq1.1}) such that $Y\in{\cal D}$, $Z\in\h$ and $K\in \s$.
\end{theorem}
\begin{proof}
We may and will assume that $\mu=0$. By \cite[Proposition
6.4]{bdhps}, for every $n\in\N$ there exists a unique solution
$(Y^n,Z^n,K^n)$  of BSDE (\ref{eq1.2}) such that $Y^n\in{\cal D}$,
$Z^n\in\h$ and $K^n\in \s$. As in the proof of Proposition
\ref{prop6} one can observe that
\[
|Y^n_\sigma|\leq E\big(|\xi|+2L^{+,*}_{T}
+\int_{0}^{T}|f(s,L^{+,*}_s)|\,ds|{\cal F}_\sigma\big),
\]
which implies that for $N>0$,
\[
E(|Y^n_\sigma|{\bf 1}_{\{|Y^n_\sigma|>N\}})\le
E\big((|\xi|+2L^{+,*}_{T}+\int_{0}^{T}|f(s,L^{+,*}_s)|\,ds){\bf
1}_{\{|Y^n_\sigma|>N\}}\big).
\]
Since by Chebyschev's inequality,
$\lim_{N\to\infty}\sup_{\sigma,n}P(|Y^n_\sigma|>N)=0$, it is clear
that
\begin{equation}
\label{eq5.5} \{Y^n_\sigma; \sigma\,\mbox{\rm stopping time},\,
\sigma\leq T,\,n\in\N\}\quad\mbox{\rm is uniformly integrable}.
\end{equation}
Now, as in the proof of Proposition \ref{prop6} one can check that
for every $\beta\in(0,1)$ there exists $C>0$ depending only on
$\mu,\lambda,T$ such that for every $n\in\N$,
\[
E\big((Y^{n,*}_T)^\beta+(\int_0^T|Z^n_s|^2\,ds)^{\beta/2}
+(K^n_T)^\beta\big)\le C
\Big(E(|\xi|+L^{+,*}_T+\int_0^T|f(s,L^{+,*}_s)|\,ds)\Big)^\beta.
\]
Moreover,  arguing as in the proof of Proposition \ref{prop9}
shows that $Y^n_t\leq Y^{n+1}_t$, $n\in\N$, $t\in[0,1]$. Therefore
for every $t\in[0,T]$ there exists $Y_t$ such that $Y^n_t\nearrow
Y_t$. By the same method as in the proof of Theorem \ref{tw1} we
can show that $Y$ is c\`adl\`ag  (the process $V$ need not be
integrable and we only know that $E(V_T)^\beta\leq
\liminf_{n\to\infty}E(V^n_T)^\beta\leq2\sup_nE(Y^{n,*}_T)^\beta$).
By Fatou's lemma,
\[
E(\int_0^T(Y_s-L_s)^-\,ds)^\beta\le
\liminf_{n\to\infty}E(\int_0^T(Y^n_s-L_s)^-\,ds)^\beta=0,
\]
which  implies that  $Y_t\geq L_t$, $t\in[0,T]$, and
$(Y^n-L)^{-,*}_T\to0$ $P$-a.s. As in the proof of Theorem
\ref{tw1} we also show that $\|Y^n-Y\|_{{\cal S}^\beta}\to0$,
$\|Z^n-Z\|_{{\cal H}^\beta}\to 0$ and $\|K^n-K\|_{{\cal
S}^\beta}\to0$, where $(Y,Z,K)$ is a solution of (\ref{eq1.1})
such that $Y, K\in\s$ and $Z\in\h$. In order to complete the proof
we have to check that $Y\in {\cal D}$, but this is an easy
consequence of (\ref{eq5.5}).
\end{proof}
\medskip

The following corollary may be proved in much the same way as
Corollary \ref{cor3}.
\begin{corollary}
\label{cor5.5} Under the assumptions of Proposition \ref{prop9},
if moreover $L=L'$ and $\xi,f,L$ and $\xi'$, $f',L$ satisfy (H3)
and (H4), $dK_s\ge dK'_s$.
\end{corollary}

We now consider reflected BSDEs of the form
\begin{equation}
\label{eq5.7} Y_t=\xi+\int_t^Tf(s,Y_s,V_s)\,ds-\int_t^TZ_s\,dW_s
+K_T-K_t\qquad t\in[0,T]
\end{equation}
and
\begin{equation}
\label{eq5.8}
Y'_t=\xi+\int_t^Tf(s,Y'_s,V'_s)\,ds-\int_t^TZ'_s\,dW_s
+K'_T-K'_t,\qquad t\in[0,T],
\end{equation}
where $V,V'$ are arbitrary progressively measurable processes on
the filtered probability space $(\Omega,{\cal F},({\cal F}_t),P)$.

\begin{proposition}
\label{prop10} Let $f$ satisfy (H2) and let $(Y,Z,K)$,
$(Y',Z',K')$ be  solutions of (\ref{eq5.7}), (\ref{eq5.8}),
respectively, such that $Y,Y'\in{\cal D}$. Then for every $p>1$
there is  $C>0$ depending only on $\mu,T$ such that
\[
\|Y-Y'\|_{{\cal S}^p}+\|Z-Z'\|_{{\cal H}^p}\leq
C\|\int_0^T|f(s,Y_s,V_s)-f(s,Y_s,V'_s)|\,ds\|_p.
\]
\end{proposition}
\begin{proof}
Without loss of generality we may and will assume that $\mu=0$ and
$U=\int_0^T|f(s,Y_s,V_s)-f(s,Y_s,V'_s)|ds\in\LL^p$.  Arguing as in
the proof of Proposition \ref{prop8}, i.e. using the fact that
$Y,Y'\in{\cal D}$, stopping at
$\tau_n=\inf\{t;\int_0^t|Z_s-Z'_s|^2ds\geq n\}\wedge T$ and then
letting $n\rightarrow\infty$ we show that $|Y_t-Y'_t|\le E(U|{\cal
F}_t)$, $ t\in[0,T]$. Hence, by Doob's maximal inequality,
\begin{equation}
\label{eq5.9}
E((Y-Y')^*_T)^p)\leq C_pE(U)^p.
\end{equation}
On the other hand, by the same method as in the proof of
Proposition \ref{prop2.1} one can show that for $n\in\N$ we have
\begin{align*}
&\int_0^{\tau_n}|Z_s-Z'_s|^2\,ds =(Y_{\tau_n}-Y'_{\tau_n})^2
+2\int_0^{\tau_n}(Y_s-Y'_s)(f(s,Y_s,V_s)-f(s,Y'_s,V'_s))\,ds\\
&\qquad-2\int_0^{\tau_n}(Y_s-Y'_s)(Z_s-Z'_s)\,dW_s
+2\int_0^{\tau_n}(Y_s-Y'_s)(dK_s-dK'_s).
\end{align*}
Since $K$ is increasing only on the set $\{s:Y_s=L_s\}$ and $K'$
is increasing only on the set $\{s:Y'_s=L_s\}$,
\[
\int_0^{\tau_n}(Y_s-Y'_s)(dK_s-dK'_s)\le0.
\]
On the other hand, by (H2),
\[
\int_0^{\tau_n}(Y_s-Y'_s)(f(s,Y_s,V'_s)-f(s,Y'_s,V'_s)\,ds)\le0.
\]
By the above,
\begin{align*}
\int_0^{\tau_n}|Z_s-Z'_s|^2\,ds &=(Y_{\tau_n}-Y'_{\tau_n})^2
+2\int_0^{\tau_n}(Y_s-Y'_s)(f(s,Y_s,V_s)-f(s,Y_s,V'_s))\,ds\\
&\quad-2\int_0^{\tau_n}(Y_s-Y'_s)(Z_s-Z'_s)\,dW_s
\end{align*}
and hence
\begin{align*}
&(\int_0^{\tau_n}|Z_s-Z'_s|^2\,ds)^{p/2} \\
&\qquad\le
c_p\big(|Y_{\tau_n}-Y'_{\tau_n}|^p+((Y-Y')_{\tau_n}^*)^{p/2}(U)^{p/2}
+|\int_0^{\tau_n}(Y_s-Y'_s)(Z_s-Z'_s)\,dW_s|^{p/2}\big)\\
&\qquad\le c'_p\big(((Y-Y')_{\tau_n}^*)^{p}+(U)^{p}
+|\int_0^{\tau_n}(Y_s-Y'_s)(Z_s-Z'_s)\,dW_s|^{p/2}\big).
\end{align*}
Using the Burkholder-Davis-Gundy inequality and letting
$n\to\infty$ we conclude from the above that
\[
E(\int_0^T|Z_s-Z'_s|^2\,ds)^{p/2}\leq
C_pE\big(((Y-Y')^*_{T})^p+(U)^p\big),
\]
which together with (\ref{eq5.9}) implies the desired result.
\end{proof}
\medskip

By the arguments from the proof of the above proposition one can
obtain similar estimates for processes on arbitrary intervals
$[t,q]\subset[0,T]$.
\begin{proposition}
\label{cordod} Under the assumptions of Proposition \ref{prop10}
for every $p>1$ there is  $C>0$ depending only on $\mu,T$ such
that for every $0\le t<q\le T$,
\begin{align*}
&\|(Y-Y'){\bf 1}_{[t,q]}\|_{{\cal S}^p}+\|(Z-Z'){\bf
1}_{[t,q]}\|_{{\cal H}^p}\\
&\qquad\le C(\|Y_q-Y'_q\|_p
+\|\int_t^q|f(s,Y_s,V_s)-f(s,Y_s,V'_s)|\,ds\|_p).
\end{align*}
\end{proposition}

To deal with generators depending on $z$ we will need the
following condition introduced in \cite{bdhps}:
\begin{enumerate}
\item[(H5)]There exist constants $\gamma\geq0$, $\alpha\in(0,1)$ and a
nonnegative progressively measurable process $g$ such that
$E(\int_0^Tg_s\,ds)<+\infty$ and
\[
|f(t,y,z)-f(t,y,0)|\leq\gamma(g_t+|y|+|z|)^{\alpha},\quad
t\in[0,T], y\in\R, z\in\Rd.
\]
\end{enumerate}

\begin{theorem}
\label{tw5.8} Let assumptions (H1)--(H5) hold. Then there exists a
unique solution $(Y,Z,K)$ of the reflected BSDEs (\ref{eq1.1})
such that $Y\in{\cal D}$, $Z\in\h$ and $K\in \s$.
\end{theorem}
\begin{proof}
Our method of proof will be adaptation of the  proofs of Theorems
6.2 and 6.3 in \cite{bdhps}. Assume that $\mu=0$. First we  show
that there exist at most one solution. Let $(Y,Z,K)$, $(Y',Z',K')$
be two solutions of (\ref{eq1.1}) such that $Y,Y'\in{\cal D}$ and
$Z,Z'\in\h$. Let $p>1$ be such that $\alpha p<1$. Then
\begin{align*}
&E(\int_0^T|f(s,Y_s,Z_s)-f(s,Y_s,Z'_s)|\,ds)^p\\
&\qquad\le
(2\gamma)^pT^{p-1}E(\int_0^T(g_s+|Y_s|+|Z_s|+|Z'_s|)^{p\alpha}\,ds)
<+\infty
\end{align*}
and hence, by Proposition \ref{prop10},
$E(\int_0^T|Z_s-Z'_s|^2\,ds)^{p/2}<+\infty$. Moreover, by
Proposition \ref{prop10} and (H1),
\begin{align*}
\|Y-Y'\|_{{\cal S}^p}+\|Z-Z'\|_{{\cal H}^p}&\le
C\|\int_0^T|f(s,Y_s,Z_s)-f(s,Y_s,Z'_s)|\,ds\|_p\\
&\le C\lambda T^{1/2}\|Z-Z'\|_{{\cal H}^p}\,.
\end{align*}
Therefore, if $2C\lambda T^{1/2}\le1$ then $\|Y-Y'\|_{{\cal
S}^p}+\frac12\|Z-Z'\|_{{\cal H}^p}\le0$, which implies that $Y=Y'$
and $Z=Z'$, and consequently, $K=K'$. If $2C\lambda T^{1/2}>1$, we
divide the interval $[0,T]$ into a finite number of intervals with
mesh $\delta>0$  such that $2C\lambda\delta^{1/2}\le1$ and prove
uniqueness by induction by using Proposition \ref{cordod} instead
of Proposition \ref{prop10}.

Now we are going to prove existence of solutions. Let $Z^0=0$. By
(H5) and Theorem \ref{tw2}, for each $n\in\N$ there exists a
unique solution $(Y^n,Z^n,K^n)$ of the reflected BSDEs (with
obstacle $L$) of the form
\begin{equation}
\label{eq5.12} Y^{n}_t=\xi+\int_t^Tf(s,Y^{n}_s,Z^{n-1}_s)\,ds
-\int_t^TZ^{n}_s\,dW_s+K^{n}_T-K^{n}_t,\quad t\in[0,T]
\end{equation}
such that $Y^n\in{\cal D}$, $Z^n\in\h$ and $Y^n,K^n\in\s$. Let
$p>1$ be such that $\alpha p<1$. Then by (H5) and Proposition
\ref{prop10},
\begin{align*}
&\|Y^{n+1}-Y^n\|_{{\cal S}^p}+\|Z^{n+1}-Z^n\|_{{\cal H}^p}\le
C\|\int_0^T|f(s,Y^n_s,Z^n_s)-f(s,Y^n_s,Z^{n-1}_s)|\,ds\|_p\\
&\qquad\le2\gamma T^{(p-1)/p}
(E(\int_0^T(g_s+|Y^n_s|+|Z^n_s|+|Z^{n-1}_s|)^{p\alpha}\,ds))^{1/p}
<+\infty.
\end{align*}
Thus, $(Y^{n+1}-Y^{n})\in{\cal S}^p$, $(Z^{n+1}-Z^n)\in{\cal
H}^p$, and hence, by elementary calculations,
$(K^{n+1}-K^n)\in{\cal S}^p$. It follows that $(Y^n-Y^1)\in{\cal
S}^p$, $(Z^n-Z^1)\in{\cal H}^p$ and $(K^n-K^1)\in{\cal S}^p$,
$n\in\N$. As in the proof of uniqueness we first assume that
$2C\lambda T^{1/2}\le1$. By (H1) and Proposition \ref{prop10},
\[
\|Y^{n+1}-Y^n\|_{{\cal S}^p}+\|Z^{n+1}-Z^n\|_{{\cal H}^p}\leq
\frac12\|Z^{n}-Z^{n-1}\|_{{\cal H}^p}
\]
for $n\in\N$. Hence
\[
\|Y^{m}-Y^n\|_{{\cal S}^p}+\|Z^{m}-Z^n\|_{{\cal H}^p}\leq
2(\frac12)^{n-1}\|Z^2-Z^{1}\|_{{\cal H}^p}
\]
for all $m\geq n$. Consequently, $\{(Y^n-Y^1,Z^n-Z^1)\}_{n\in\N}$
is a Cauchy sequence in ${\cal S}^p\times{\cal H}^p$ converging to
some process $(\tilde Y,\tilde Z)$. Since $(Y^1,Z^1)\in{\cal
D}\times\h$, it follows that
\begin{equation}
\label{eq5.09} Y^n\to Y=\tilde Y+Y^1\mbox{ \rm in }{\cal D},\quad
Z^n\to Z=\tilde Z+Z^1\,\,\mbox{\rm in }{ \cal H}^\beta,
\,\beta\in(0,1).
\end{equation}
Using standard arguments one can also check that for every
$\beta\in(0,1)$ the sequence $\{K^n\}$ converges in ${\cal
S}^\beta$ to some nondecreasing continuous process $K$ and that
$(Y,Z,K)$ is a solution of the reflected BSDE (\ref{eq1.1}).

If $2C\lambda T^{1/2}>1$ we consider a partition
$0=t_0<t_1<\dots<t_k=T $ of the interval $[0,T]$ such that
$t_{i}-t_{i-1}\le\delta$, $i=1,\dots,k$, and
$2C\lambda\delta^{1/2}\le1$. By  arguments from the first part of
the proof and Proposition \ref{cordod},
\begin{align*}
&\|(Y^{n+1}-Y^n){\bf 1 }_{[t_{k-1},T]}\|_{{\cal S}^p}
+\|(Z^{n+1}-Z^n){\bf }_{[t_{k-1},T]}\|_{{\cal H}^p}\\
&\qquad\quad\le
C\|\int_{t_{k-1}}^{T}
|f(s,Y^{n}_s,Z^{n}_s)-f(s,Y^n_s,Z^{n-1}_s)|\,ds\|_p\\
& \qquad\quad\le\frac12\|(Z^n-Z^{n-1}){\bf
1}_{[t_{k-1},T]}\|_{{\cal H}^p}
\le(\frac12)^{n-1}\|(Z^2-Z^{1}){\bf 1}_{[t_{k-1},T]}\|_{{\cal
H}^p},
\end{align*}
which implies that
\[
\|(Y^{m}-Y^n){\bf 1}_{[t_{k-1},T]}\|_{{\cal S}^p}
+\|(Z^{m}-Z^n){\bf 1}_{[t_{k-1},T]}\|_{{\cal H}^p} \le
2(\frac12)^{n-1}\|(Z^2-Z^{1}){\bf 1}_{[t_{k-1},T]}\|_{{\cal H}^p}
\]
for all $m\geq n$. Accordingly, $\{(Y^n-Y^1){\bf
1}_{[t_{k-1},T]},(Z^n-Z^1){\bf 1}_{[t_{k-1},T]}\}_{n\in\N}$ is a
Cauchy sequence in ${\cal S}^p\times{\cal H}^p$. Therefore as in
the proof of (\ref{eq5.09}) one can show that there exist
processes $Y^{(k)},Z^{(k)}, K^{(k)}$ such that $ Y^n{\bf
1}_{[t_{k-1},T]}\to Y^{(k)}$ in ${\cal D}$, $Z^n{\bf
1}_{[t_{k-1},T]}\to Z^{(k)}$ in ${ \cal H}^\beta$,
$\beta\in(0,1)$, and  $ (K^n-K^n_{t_{k-1}}){\bf
1}_{[t_{k-1},T]}\to K^{(k)}$ in ${ \cal S}^\beta$,
$\beta\in(0,1)$. Observe that $Y^{(k)}_T=\xi$ and
\begin{equation}
\label{eqdo0} Y^{(k)}_t=Y^{(k)}_{{t_{k-1}}}-\int_{t_{k-1}}^t
f(s,Y^{(k)}_s,Z^{(k)}_s)\,ds+\int_{t_{k-1}}^t
Z^{(k)}_s\,dW_s-K^{(k)}_t,\quad t\in[t_{k-1},T].
\end{equation}
By (H5) and Theorem \ref{tw2}, for each $n\in\N$ there exists a
unique solution $(Y^n,Z^n,K^n)$ of the reflected BSDEs
(\ref{eq5.12}) with $[0,T]$ replaced by $[0,t_{k-1}]$ and $\xi$
replaced by $Y^{(k)}_{t_{k-1}}$. Therefore in the same manner as
before we can see that there exist processes $Y^{(k-1)}$,
$Z^{(k-1)}$,  $K^{(k-1)}$ such that $ Y^n{\bf
1}_{[t_{k-2},t_{k-1}]}\to Y^{(k-1)}$ in ${\cal D}$, $ Z^n{\bf
1}_{[t_{k-2},t_{k-1}]}\to Z^{(k-1)}$  in  ${ \cal H}^\beta$,
$\beta\in(0,1)$,  and  $ (K^n-K^n_{t_{k-2}}){\bf
1}_{[t_{k-2},t_{k-1}]}\to K^{(k-1)}$ in ${ \cal S}^\beta$,
$\beta\in(0,1)$. We continue in this fashion to obtain for
$i=k-1,...,1$ the triple of processes $(Y^{(i)},Z^{(i)},K^{(i)})$
such that $Y^{(i)}_{t_i}=Y^{(i+1)}_{t_i}$  and
\begin{equation}
\label{eqdo1} Y^{(i)}_t=Y^{(i)}_{{t_{i-1}}}-\int_{t_{i-1}}^t
f(s,Y^{(i)}_s,Z^{(i)}_s)\,ds+\int_{t_{i-1}}^t Z^{(i)}_s\,dW_s
-K^{(i)}_t,\quad t\in[t_{i-1},t_i].
\end{equation}
It is clear that for $i=1,\dots,k$,
\begin{equation}\label{eqdo-2} L_t\leq Y^n_t\arrowp Y^{(i)}_t\geq L_t,\quad
t\in[t_{i-1},t_i]\end{equation}  and
\begin{equation}
\label{eqdo-1}0=\int_{t_{i-1}}^{t_i}
(Y^n_t-L_t)\,dK^n_t\arrowp\int^{t_i}_{t_{i-1}}(Y^{(i)}_t-L_t)\,d
K^{(i)}_t=0.
\end{equation}
Set $Y_T=\xi$, $Z_T=0$ and
\[
Y_t=Y^{(i)}_t,\quad Z_t=Z^{(i)}_t,\quad t\in[t_{i-1},t_i),\,\,
i=1,\dots,k,\quad K=\sum_{i=1}^k K^{(i)},
\]
and observe that $Y,Z,K$ are progressively measurable, $Y\in{\cal
D}$, $Z\in{ \cal H}^\beta$, $\beta\in(0,1)$, and $ K\in{ \cal
S}^\beta$, $\beta\in(0,1)$, $K_0=0$. Moreover, the process $Y$ is
continuous, $K$ is continuous and nondecreasing, and by
(\ref{eqdo0}), (\ref{eqdo1})  the triple $(Y,Z,K)$ satisfies the
forward equation
\[
Y_t=Y_{0}-\int_{0}^t
f(s,Y_s,Z_s)\,ds+\int_{0}^t Z_s\,dW_s-K_t,\quad t\in[0,T].
\] Since
$Y_T=\xi$, it satisfies the backward equation(\ref{eq1.1})${}_1$
as well. Finally, by (\ref{eqdo-2}), $Y_t\geq L_t$, $t\in[0,T]$,
whereas by (\ref{eqdo-1}),
\[
\int_0^T(Y_t-L_t)\,dK_t
=\sum_{i=1}^k\int^{t_i}_{t_{i-1}}(Y^{(i)}_t-L_t)\,d
K^{(i)}_t=0,
\]
i.e. (\ref{eq1.1})${}_2$ and (\ref{eq1.1})${}_3$ are satisfied.
Thus, the triple $(Y,Z,K)$ is a solution of the reflected BSDE
(\ref{eq1.1}).
\end{proof}

\begin{remark}
{\rm Similarly to the case $p>1$ (see Remark \ref{rem4.6}), if we
assume that stronger than (H2) condition (\ref{eq4.16}) is
satisfied, then in Theorem \ref{tw2} , Corollary \ref{cor5.5} and
Theorem \ref{tw5.8}  assumption (H4b) may be omitted.}
\end{remark}
{\bf Acknowledgements}\\
Research supported by Polish Ministry of Science and Higher
Education (grant no. N N201 372 436).
\medskip\\
We thank the anonymous referees for careful reading of the first
version of the paper and for valuable comments.

\end{document}